\documentclass[12pt]{article}		
\usepackage{tabls,multirow}  
\usepackage{amssymb,amsmath,amsfonts}
\usepackage{color,url,hyperref}
\usepackage{graphicx}
\usepackage{verbatim} 
\usepackage{caption}
\usepackage{subcaption}
\usepackage[ruled,vlined,linesnumbered]{algorithm2e}

\definecolor{dred}{rgb}{0.6,0.0,0.0}

\def\TT{\mathcal{T}}
\def \MATLAB {\textsc{matlab}}
\def\trace{\mbox{\tt tr}}

\newtheorem{lemma}{Lemma}[section]

\newtheorem{proposition}{Proposition}[section]
\newtheorem{figr}{Figure}[section] 

\def\myrefs#1#2{ 
{\bigskip \noindent
{\Large \bf #2}  
 \list {[\arabic{enumi}]}{\settowidth\labelwidth{[#1]}
 \leftmargin\labelwidth 
 \advance\leftmargin\labelsep
 \usecounter{enumi} }  
 \def\newblock{\hskip .11em plus .33em minus .07em}
 \sloppy\clubpenalty4000\widowpenalty4000
 \sfcode`\.=1000\relax}  }

\newenvironment{proof}{\begin{trivlist}
                       \item[]{\bf Proof.}
                       \hspace{0cm} }{\hfill $\Box$
                       \end{trivlist}}

\newcommand{\ab}{[a, \ b]}

\def\inv{^{-1}}%
\def\nref#1{(\ref{#1})}

\def\comb#1,#2,{ \left( {#1 \atop #2 } \right)  }%
\def\prodd#1,#2,#3,{ \prod_{\scriptstyle #1 \atop\scriptstyle #2 }^{#3} }%
\def\summ#1,#2,#3,{ \sum_{\scriptstyle #1 \atop\scriptstyle #2 }^{#3} }%
\def\upp#1{^{({#1})}} %

\newtheorem{algor}{{\sc Algorithm}}[section]
\newtheorem{tabl}{Table}[section]

\def\betab{\begin{tabbing} 
xxxx\=xxxx\=xxx\=xx\=xx\=xx\=xx\=xx\=xx\=xx\=xx\=xx\=xx\= \kill} 
\def\entab{\end{tabbing}\vspace{-0.12in}}

\newcommand{\eq}[1]{\begin{equation}\label{#1}}
\newcommand{\en}{\end{equation}}

\newcommand{\beeq}[1]{\begin{equation}\label{#1}}
\newcommand{\eneq}{\end{equation}}

\title{Efficient estimation of eigenvalue counts in an interval}
\author{Edoardo Di Napoli
\thanks{J\"ulich Supercomputing Centre, Forschungszentrum J\"ulich, D-52425 J\"ulich. 
        {\tt e.di.napoli@fz-juelich.de}.
Work of this author was supported by VolkswagenStiftung,
       the Excellence Initiative of the German federal and state governments 
       and the J\"ulich Aachen Research Alliance -- High-Performance Computing.}
\and
Eric Polizzi
\thanks{Department of Electrical and Computer Engineering,
        University of Massachusetts, Amherst. 
{\tt polizzi@ecs.umass.edu}. Work of this author was supported by the
       National Science Foundation under Grant \#ECCS-0846457.}
\and
Yousef Saad 
\thanks{Computer Science \& Engineering,
        University of Minnesota, Twin Cities. 
{\tt saad@cs.umn.edu}. 
Work of this author 
was supported by        the Scientific Discovery through Advanced Computing (SciDAC) program funded by
       U.S. Department of Energy, Office of Science, Advanced Scientific Computing Research and Basic Energy
       Sciences DE-SC0008877}
} 

\date{\today} 

\begin{document} 

\maketitle 

\begin{abstract}
  Estimating the number of eigenvalues  located in a given interval of
  a large sparse  Hermitian matrix is an important  problem in certain
  applications and  it is  a prerequisite of  eigensolvers based  on a
  divide-and-conquer paradigm.  Often an  exact count is not necessary
  and methods based  on stochastic estimates can be  utilized to yield
  rough approximations.   This paper  examines a number  of techniques
  tailored to this specific  task.  It reviews standard approaches and
  explores  new ones  based on  polynomial and  rational approximation
  filtering combined with a stochastic procedure.
\end{abstract} 


\section{Introduction}
Recent efforts to develop alternative eigensolvers
~\cite{Schofield-al-2011,FEAST,SakSig03} for large scale scientific
applications rely on ``splitting'' the spectrum of an eigenproblem in
intervals and extracting eigenpairs from each one independently. In
order to be efficient, this strategy requires an approximate knowledge
of the number of eigenvalues included in each of these intervals.  In
general, any algorithm based on a form of subspace iteration for
computing eigenvalues in a given interval benefits from the knowledge
of the approximate number of eigenvalues inside that interval so as to
select the dimension of the subspace to use in the iteration. In this
context subspace size selection was already addressed by Sakurai and
co-workers,
\cite{FutamuraSakuraiEigCnt,Senzaki-al-eigCnt-2010,Maeda-al-eigCnt-2011},
using a few different methods.  An approximate count of the
eigenvalues located in an interval can also help estimate the rank of
a matrix and this is needed in the context of sampling-based methods
\cite{halko-al-survey-11,Mahoney-book}.  The goal of this paper is to
explore inexpensive algorithms for determining the number of
eigenvalues of a Hermitian matrix that are located in a given
interval.

The standard way of computing the number of eigenvalues 
of a Hermitian matrix $A$ located inside
an interval $\ab$ is to resort to the Sylvester law of
inertia~\cite{GVL-book}.  For the sake
of brevity this introduction will only  discuss the standard eigenvalue problem
although the paper deals with both standard and
generalized problems. 
If $A$ is nonsingular, it admits the
decomposition $A = L D L^T$, where $L$ is unit lower triangular, and
$D$ is diagonal. The Sylvester inertia theorem then states that the
inertias of $A$ and $D$ are the same. This means that the number of
eigenvalues of $A$ that are positive is the same as the number of
positive entries in the diagonal of $D$ (Sturm count). Thus, the
$LDL^T$ factorizations for the shifted matrices $A - a I$ and $A - b
I$ (assuming that these exist) yield respectively the number of
eigenvalues larger than $a$ and $b$. The difference between these two
numbers gives the eigenvalue count $\mu_{\ab}$ in $\ab$. While
this method yields an exact count, it requires two complete $LDL^T$
factorizations and this can be quite expensive for realistic
eigenproblems.

This paper discusses two alternative methods which provide only an
estimate for $\mu_{\ab}$ but which are relatively inexpensive. Both
methods work by estimating the trace of the spectral projector $P$
associated with the eigenvalues inside the interval $\ab$.  This
spectral projector is expanded in two different ways and its trace is
computed by resorting to stochastic trace estimators, see,
e.g.,\cite{Hutchinson-est,JTangYS}.  The first method utilizes
filtering techniques based on Chebyshev polynomials. The resulting
projector is expanded as a polynomial function of $A$.  In the second
method the projector is constructed by integrating the resolvent of
the eigenproblem along a contour in the complex plane enclosing the
interval $\ab$. In this case the projector is approximated by a
rational function of $A$.  

For each of the above methods we present various implementations
depending on the nature of the eigenproblem (generalized vs standard),
and cost considerations.  Thus, in the polynomial expansion case, we
propose a barrier-type filter when dealing with a standard
eigenproblem, and two high/low pass filters in the case of generalized
eigenproblems. In the rational expansion case we have the choice of
using an LU factorization or a Krylov subspace method to solve linear
systems. The optimal implementation of each method used for the
eigenvalue count depends on the situation at hand and involves
compromises between cost and accuracy. 
While it is not the aim of this
paper to explore detailed analysis of these techniques, we will
discuss various possibilities and provide illustrative examples.

The polynomial and rational expansion methods are motivated by 
two distinct approaches recently suggested in the context of
electronic structure calculations: i) spectrum slicing and ii) Cauchy
integral eigen-projection. In the spectrum slicing techniques
\cite{Schofield-al-2011} the eigenpairs are computed by dividing the
spectrum in many small subintervals, called `slices' or `windows'.
For each window a barrier function is approximated by
Chebyshev-Jackson polynomials in order to select only the portion of
the spectrum in the slice. In this method, it is important to
determine an approximate count of the eigenvalue in each sub-interval
in order to balance the calculations in a parallel
implementation. 

The second set of methods is based on eigen-projectors
expressed in the form of Cauchy integrals~\cite{FEAST,SakSig03}.
They essentially compute an orthonormal basis of the invariant
eigenspace ${\cal V}$ associated with the eigenvalues located in the
interval. For these methods to work efficiently one must have a good
idea of the dimension of the subspace. This dimension must not be
smaller than that of ${\cal V}$ if we are to account for all the
eigenvalues inside the interval $\ab$ and, for reasons related to
computational costs, it should also not be too large.

The paper is organized as follows.
Section~\ref{sec:eigcount} introduces the eigenvalue count problem
and gives an overview of traditional approaches for solving it.
Section~\ref{sec:polynomial} discusses methods based on
polynomial expansions and 
Section~\ref{sec:rational} is devoted to methods based on
rational function expansions. Section~\ref{sec:numtests} 
presents additional issues and provides a series of numerical
tests to illustrate the behavior of  the methods.
Finally, Section~\ref{sec:concl} offers some concluding remarks.

\section{Eigenvalue counts}
\label{sec:eigcount}
Let $\lambda_j,\ j=1,\cdots,n$ be the eigenvalues, labeled by
increasing value, and $u_1, u_2, \cdots, u_n$ the associated
orthonormal eigenvectors of an Hermitian matrix $A$ (the generalized
problem will be discussed later). Assuming that $\lambda_1 \le a < b
\le \lambda_n$, our aim is to count the number of eigenvalues
$\lambda_i$ in the interval $\ab$. As described in the introduction,
the standard way of obtaining this count is to resort to the Sylvester
inertia theorem which will require two $LDL^T$ factorizations. Since
exact factorizations can be computationally expensive, choosing the
correct implementation of the $LDL^T$ factorization is crucial.

In modeling Hamiltonians of 2-dimensional (2-D) physical systems using
finite differences sparse factorizations can be quite effective so
that an eigenvalue count based on the inertias may be the method of
choice. However, this approach becomes expensive in realistic cases
where the matrix arises from simulations of 3-dimensional (3-D)
phenomena. As is well-known \cite{George-Liu-book,TDavis-book}, in the
3-D case the factorization becomes very costly both in terms of
storage and arithmetic, due to the amount of fill-in generated.  For
dense eigenproblems the number of floating point operations per
factorization is of order $O(n^3)$ and this becomes prohibitive for
large matrices. Hence, counting eigenvalues based on the inertia
theorem is a viable method only when dealing with fairly small dense
matrices or for sparse matrices whose factorization is not too costly,
e.g., those generated by 2-D models.

The problem of eigenvalue counts is also closely related to that of
computing ``Density of States'' (DoS) a term used by physicists for
the `spectral density' or the probability of finding an eigenvalue at
a given point in the real line.  Some of these techniques bear some
similarity with the ones described here \cite{LinYangSaad13-TR}. For
example, one can view the polynomial-based techniques presented in
this paper as an adaptation of the Kernel Polynomial Method for
computing the DoS to the problem of estimating eigenvalue counts
\cite{SilverRoederAl,Wang-DOS,Roder-al-96,SilverRoder1994}.

This paper explores two alternative approaches that compute an
estimate of the eigenvalue count in the interval $\ab$ by seeking an
approximation to the trace of the eigen-projector: \eq{eq:proj} P =
\sum_{\lambda_i \ \in \ \ab} u_i u_i^T .  \en The eigenvalues of a
projector are either zero or one and so the trace of $P$ is equal to
the number of terms in the sum \nref{eq:proj}, i.e.~to the number of
eigenvalues in $\ab$.  Therefore, we can calculate the number of
eigenvalues $\mu_{\ab}$ located in the interval $\ab$ by evaluating
the trace of the related projector \nref{eq:proj}:
\[
\mu_{\ab} = \trace (P ) \ .
\]
If $P$ were available explicitly, we would be able to compute its
trace directly and obtain $\mu_{\ab}$ exactly.  The projector $P$ is
typically not available in practice but it is possible to
inexpensively approximate it in the form of either a polynomial or a
rational function of $A$. To this end, we can interpret $P$ as a step
function of $A$, namely:
\begin{eqnarray}
\label{eq:hoft} 
P = h(A) \quad \mbox{where} \quad h(t) = \left\{\begin{array}{l l}
1  & \quad \textrm{if}\ \ t \ \in \ \ab\\
0  & \quad \textrm{otherwise}\\
\end{array}\right. .
\end{eqnarray}
One can now approximate $h(t)$ with either a finite sum $\psi(t)$ of
Chebyshev polynomials or a closed line integration of a rational
function $\chi(t)$ on the complex plane. These two approaches lead to
distinct approximations of the projector $P$, namely $P \approx
\psi(A)$ or $P \approx \chi(A)$.  In this form, it becomes possible to
estimate the trace of $P$ by a so-called stochastic estimator 
developed by Hutchinson~\cite{Hutchinson-est} and further improved
more recently~\cite{JTangYS, Iitaka:705693, Wong:2004tm}.

Hutchinson's unbiased estimator uses only matrix-vector products to
approximate the trace of a generic matrix $A$.  The idea is based on
the use of identically independently distributed (i.i.d.) Rademacher
random variables whereby each entry of randomly generated vectors $v$
assumes the values $-1$ and $1$ with equal probability ${1 \over 2}$.
Hutchinson proved in a lemma that $E(v^\top A v) = \trace (A)$.  Thus,
an estimate $\TT_{n_v}$ of the trace $\trace (A)$ can be obtained by
generating $n_v$ samples of random vectors $v_k,\ k=1,..,n_v$ and
computing the average of $v_k^\top A v_k $ over these samples.

\begin{equation}
\label{eq:trace_estimator}
\trace (A) \approx \TT_{n_v} =  {1 \over n_v}  \sum_{k=1}^{n_v} v_k^\top A v_k.
\end{equation}

In practice there is no need to take vectors with entries equal to
Rademacher random variables. Any sequence of random vectors $v_k$
whose entries are i.i.d.~random variables will do as long as the mean
of their entries is zero~\cite{Bekas-al-DIAGEST}\footnote{This form
  was used by physicists to compute the density of
  states~\cite{SilverRoederAl,Wang-DOS,Roder-al-96,SilverRoder1994}.}.
For example one can use normally distributed variables and define the
Gaussian estimator exactly in the same fashion as in
\eqref{eq:trace_estimator}. While the variance of such an estimator is
larger than that of Hutchinson, which uses Rademacher vectors, it
shows a better convergence to the trace, in terms of the number of
sample vectors $n_v$~\cite{Avron:2011hg}.

Instead of the variance of the estimator, a more meaningful criterion
is a less than $\delta$ probability that the estimator computes a
value $\TT_{n_v}$ for $\trace (A)$ whose relative error exceeds
$\epsilon$, i.e.
\begin{equation}
{\rm Pr}\left(|\TT_{n_v} - \trace(A)| \ge \epsilon \ \trace(A)\right) \le \delta.
\end{equation}
In the particular case the matrix $A$ is a projector $P$ its trace
assumes only integers values and the rounded off value of
$\TT_{n_v}$ will be indistinguishable from $\trace(P)$ when
their difference is less then ${1 \over 2}$. In this particular case
$\epsilon \equiv 1 / (2\ \trace(P))$ and the following lemma
holds~\cite{Avron:2011hg}
\begin{lemma}
\label{th:toledo}
Let be $P \in \mathbb{R}^{n \times n}$ a projection matrix and let $\delta > 0$ be a failure probability. Then for $n_v \ge 16 \ \trace(P) \ln ({2 \over \delta})$, the estimator $\TT_{n_v}$ of $P$ satisfies
\[
{\rm Pr}\left(\lfloor \TT_{n_v}\rceil \neq \trace(P)\right) \le \delta.
\]
\end{lemma}

The lemma above suggests that, even for $\delta = 1$, the minimum
number of sample vectors is of the order of magnitude of the trace of
the projector with a large pre-factor. In practice the above lemma
gives a rather loose bound and in many cases the estimator
$\TT_{n_v}$ converges for much lower values of $n_v$ than
indicated by the lemma. In Sec.~\ref{sec:vectrace} we will show
examples for which there exists a value $n_v \lesssim \trace(P)$ after which
the estimator $\TT_{n_v}$ does not experience any relevant
variation. In fact the source of error due to the use of a trace
estimator is quite often negligible relative to  the bias
introduced by approximating the projector $P = h(A)$ with either
$\psi(A)$ or $\chi(A)$.

We can now compute the trace of $P$ as:
\begin{eqnarray}
\label{eq:expan}
\mu_{\ab} \approx \left\{
  \begin{array}{l l}
    \displaystyle {n \over n_v}  \sum_{k=1}^{n_v} v_k^\top \psi(A) v_k & \quad \textrm{Polynomial expansion filtering}\\
    \displaystyle {n \over n_v}  \sum_{k=1}^{n_v} v_k^\top \chi(A) v_k & \quad \textrm{Rational expansion filtering}.
  \end{array}
\right.
\end{eqnarray}
where the sample vectors of the Gaussian estimator are normalized to
one $\|v_k\|=1$. This constraint introduces a factor $n$ in the
estimator but leaves the conclusion of lemma \ref{th:toledo}
unchanged. Moreover it avoids those rare cases when the Gaussian
quotient returns rather large values. We will refer to this modified
definition of $\TT_{n_v}$ as Rayleigh Quotient (RQ) estimator.

The polynomial expansion approach does not require any factorization
of $A$, and this is a big advantage when $A$ is large. Formally, the
rational expansion approach would require a few such factorizations,
one for each pole $z_i$ of the rational function. However, an exact
factorization is no longer needed since we only need to solve with low
accuracy linear systems with matrices of the form $A -z_i I$ and
preconditioned iterative methods can be invoked for this purpose.  In
the next section we illustrate both approximations in detail and show
that the error in computing $\mu_{\ab}$ depends essentially on how
well the projector is approximated by the expansion and may be
strongly affected by the presence of clustered eigenvalues close to
the limits of the interval $\ab$.


\section{Polynomial expansion filtering} 
\label{sec:polynomial}

In the polynomial filtering approach, the step function $h(t)$ in
\eqref{eq:hoft} is expanded into a degree $p$ Chebyshev polynomial
series: 
\eq{eq:polH} h(t) \approx \psi_p(t) = \sum_{j=0}^p \gamma_j
T_j(t) .  
\en 
Here $T_j$ are the $j$-degree Chebyshev polynomials of
the first kind, and the coefficients $\gamma_j$ are the expansion
coefficients of the step function $h$ which are known to be
\eq{eq:gammaj}
\gamma_{j}=
\left\{ \begin{array}{rcl}
{\displaystyle \frac{1}{\pi}\left( \arccos(a)-\arccos(b)\right) } & : & j = 0, \\
\ \\
{\displaystyle 
\frac{2}{\pi}\left( \frac{\sin(j\arccos(a)) - \sin(j\arccos(b))}
{j}\right)} & : & j > 0 . 
\end{array} \right.
 \en
As a result we obtain an expansion of $P$ into matrices $T_j(A)$
\eq{eq:polP}
P \approx \psi_p(A) = \sum_{j=0}^p \gamma_j T_j(A). 
\en
The above derivation is based on the standard assumption that all the
eigenvalues of $A$ lie in the interval $[-1, \ 1]$ but it can be
trivially extended to a generic spectrum with a simple linear
transformation that maps $[\lambda_1, \ \lambda_n ]$ into $[-1, \ 1]$.
This linear transformation is: 
\[
l(t) = \frac{t - (\lambda_n + \lambda_1)/2}{ (\lambda_n - \lambda_1)/2} 
\]
and it 
requires estimates of the largest and smallest eigenvalues $\lambda_n, 
\lambda_1$. For the scheme to work it is necessary that
the estimate for $\lambda_n $  be larger than 
$\lambda_n $ and 
the estimate for $\lambda_1 $  be smaller than 
$\lambda_1 $.

Two  examples of a Chebyshev expansion of $h$ are shown (in red) in
Fig.~\ref{fig:midP}.  As can be observed from the plots (red curves),
the expansion of $h(t)$ has harmful oscillations near the
boundaries. These are known as \emph{Gibbs oscillations}. To alleviate
this behavior it is customary to add damping multipliers -- Jackson
coefficients -- so that \nref{eq:polP} is actually replaced by
\eq{eq:polPJ} 
P \approx \psi_p(A) = \sum_{j=0}^p g_j^{p} \gamma_j
T_j(A) .  
\en 
Notice that the matrix polynomial for the standard
Chebyshev approach has the same expression as above with
the Jackson coefficients $g_j^p$ all set to one, so we will use
the same symbol to denote both expansions.
The Jackson coefficients in their original form can be shown 
to be given by the formula,
\eq{eq:JacCoef} 
g_{j}^{p}=
\frac{ \left( 1 - \frac{j}{p+2} \right) \sin( \alpha_{p} ) \cos( j \alpha_{p} )\
 +
 \frac{1}{p+2} \cos( \alpha_{p} ) \sin( j \alpha_{p} ) }{ 
\sin( \alpha_{p} ) }
\quad \mbox{where}\quad
 \alpha_{p}=\frac{\pi}{p+2}
 , 
\en 
which was developed in \cite{Jay-al}. Note that we can also write these
coefficients in a slightly shorter form as:
\eq{eq:JacCoef1}
g_{j}^{p} = 
\frac{\sin(j+1)\alpha_p }{(p+2)\sin \alpha_p } + 
 \left(1 - \frac{j+1}{p+2} \right) \cos(j \alpha_{p} ) .
\en
\begin{figure}[htb]
\centering
  \begin{subfigure}[b]{0.48\textwidth}
                \centering
                \includegraphics[width=\textwidth,height=0.808\textwidth]{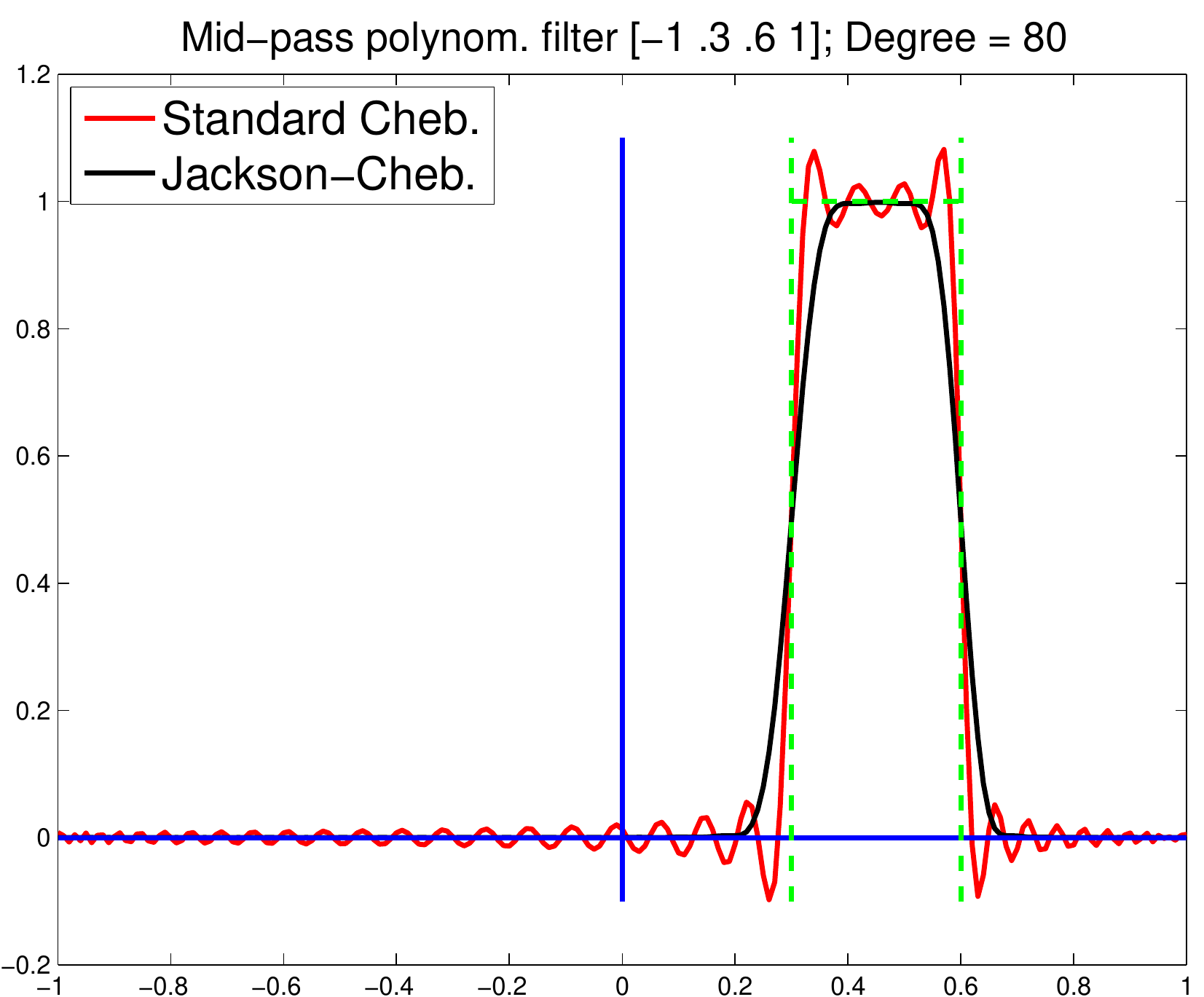}
                \caption{Mid-pass filter of degree $80$.}
                \label{fig:mid-pass}
  \end{subfigure}
  \begin{subfigure}[b]{0.48\textwidth}
                \centering
                \includegraphics[width=\textwidth,height=0.808\textwidth]{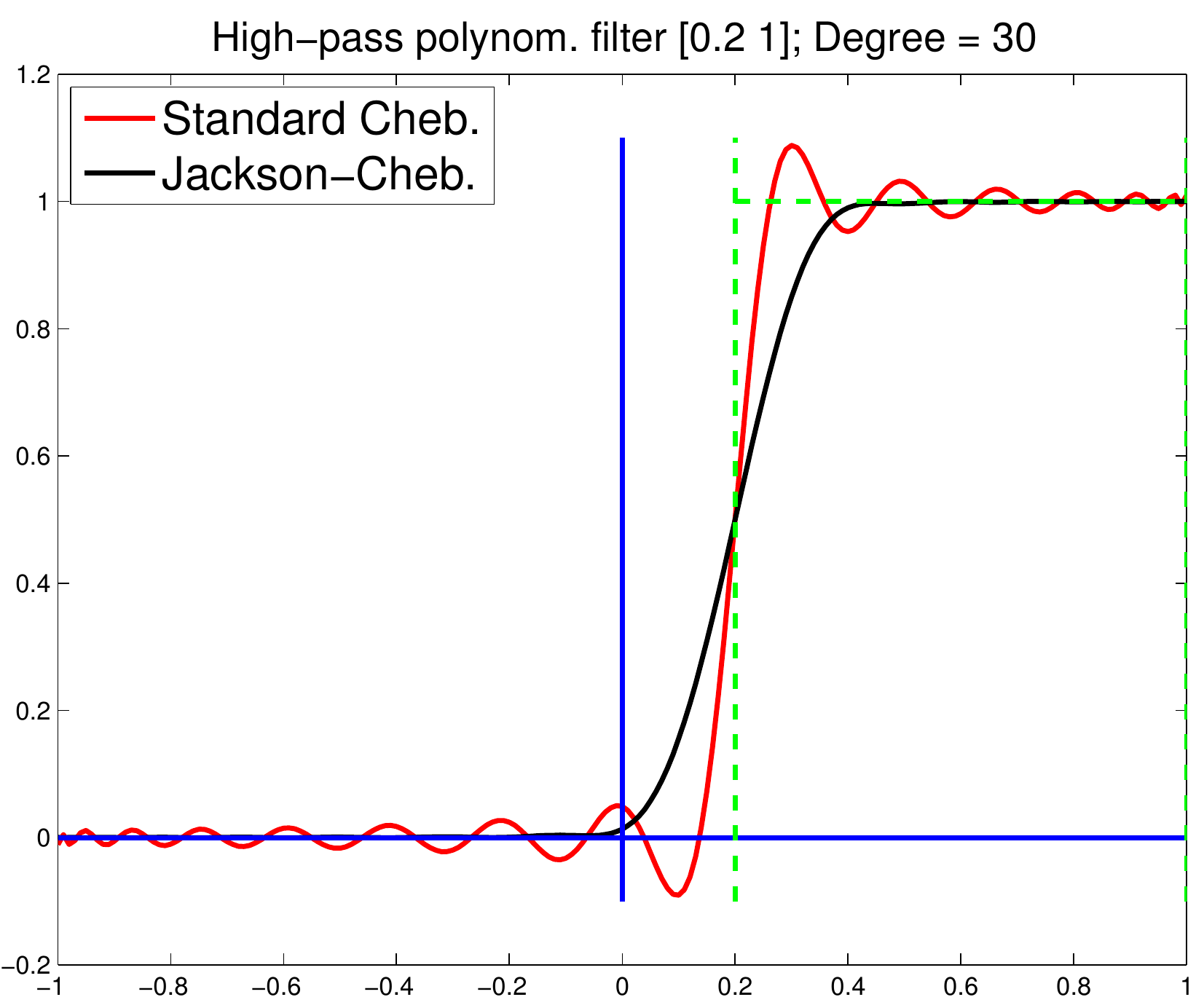}
                \caption{High-pass filter of degree $30$.}
                \label{fig:high-pass}
  \end{subfigure}
\caption{Two examples of polynomial filters.} 
\label{fig:midP}
\end{figure} 

Substituting  the expression for $\psi_p(A)$ directly into the
stochastic estimator~(\ref{eq:expan}), yields the following estimate
\eq{eq:estP}  
\mu_{\ab} = \trace (P) \approx \frac{n}{n_v} \sum_{k=1}^{n_v} 
\left[
\sum_{j=0}^p \gamma_j v_k^T T_j(A) v_k \right]   
 . \en
A clear  advantage of this approach is that it requires only matrix-vector
products. In addition, the vectors $w_j = T_j(A)v$ for a given $v$ can be
easily computed using the 3-term recurrence relation of Chebyshev polynomials
$T_{j+1} (t) = 2\, t\, T_j(t) - T_{j-1} (t)$ which leads to 
\[ 
w_{j+1} = 2\, A\, w_j - w_{j-1}. 
\]
A similar method which uses a more complicated expansion into
orthogonal polynomials, was also advocated in \cite{saad-filt}.

\subsection{Theoretical considerations}\label{sec:Analysis}
It is a known fact that the Chebyshev expansion of a function $h(t)$
(defined in $[-1,\ 1]$) is essentially a Fourier expansion of a
function obtained from $h$ by the change of variables $\cos \theta =
t$ for $\theta \ \in \ [0, \ \pi]$.  The function $h(t)$ becomes
$h(\cos \theta)$ for $0 \le \theta \le \pi$ and this function is
extended into the interval $(\pi, \ 2 \pi]$ by symmetry so the result
is an even function (see \cite[sec.  2.4]{Rivlin} for details). When
the function to expand is discontinuous, as is the case here, the
Fourier expansion does not converge uniformly. In fact the function
will oscillate around the discontinuity, and the maximum of the
polynomial in $(-1, \ 1)$ converges to a number strictly larger than
one \cite{Angot}.  We do however, have pointwise convergence to
$(h(x+0)+h(x-0))/2$ where $h(x+0) $ (resp. $h(x-0))$ represents the
limit of $h(t)$ when $t$ converges to $x$ from the right (resp. left).

The Jackson expansion is a member of several techniques 
used to get rid of Gibbs oscillations. 
It can be viewed as the expansion of a 
sequence of smoothed versions of the original
function. 
Another form of smoothing proposed by Lanczos~\cite[Chap. 4]{Lanczos-book} and
referred to as $\sigma$-smoothing,
uses the   simpler damping coefficients, called 
 $\sigma$ factors by the author:
\[
\sigma_0^p =1; \quad 
\sigma_j^p = \frac{ \sin ( j \theta_p ) }{j \theta_p}, \quad
j=1,\cdots,p \quad \mbox{with} \   
\theta_p =\frac{\pi}{p+1} . 
\] 

The damping factors are small for larger values of $j$ and this has the
effect of reducing the oscillations. The Jackson coefficients have a much
stronger damping effect on these  last terms than the Lanczos $\sigma$ factors.
 For example the 
very last factors, and their approximate values for large $p$'s,  are  in each case:
\[
g_{p}^{p} = 
\frac{2 \sin^2 (\alpha_p) }{p+2} \approx \frac{2 \pi^2}{(p+2)^3}  ;
\qquad
\sigma_p^p = 
\frac{ \sin (p\ \theta_p ) }{p\ \theta_p}\approx
 \frac{1}{p} .
\] 

The above discussion underscores the difficulty in analyzing
convergence and in selecting a proper polynomial degree.  First we do
not have uniform convergence for the case of the standard Chebyshev
approximation and therefore it is difficult to relate the vectors
$\psi_p(A)v$ to $h(A)v$ as would be required to obtain error bounds.
Second, for the Jackson (and Lanczos $\sigma$-) smoothing the function
being approximated is no longer just $h$ but a smoothed intermediate
version $h^{(p)}$ which varies with the degree $p$ selected.  So the
process in this case has two steps: First smooth $h$ into $h^{(p)}$
then approximate $h^{(p)}$.

However, it is important to remember that the standard Chebyshev is
the best approximation of the original function $h$. So any
approximation to $h$ that is used, including the Jackson and
Lanczos-$\sigma$ expansion, will have an error in the weighted $L_2$
norm that cannot be smaller that that of the standard Chebyshev
approach.  In the following we assume that the interval of interest is
$[-1, \ 1]$ and denote by $\| f(t)\|_2$ the $L_2$ norm associated with
the Chebyshev inner product of functions: \eq{eq:ChebP} \left\langle
  \psi_p, \psi_q \right\rangle = \int_{-1}^{1} \frac{\psi_p(s)
  \psi_q(s)}{\sqrt{1 - s^2}}ds \en We can now state the following
result.
\begin{proposition}\label{lem:ineq}
Let $\psi_p\upp{C}$, $\psi_p\upp{J}$ and $\psi_p\upp{L}$  
be the $p$-th degree Chebyshev,  Jackson, or
$\sigma$-Lanczos polynomial approximations to the function $h$
respectively, and let $\gamma_j$ be the sequence defined by 
\eqref{eq:gammaj} 
Then, 
\eq{eq:ineq1}
\| \psi_p\upp{C} -  h(t) \|_2 \le \| \psi_p\upp{J} -  h(t) \|_2, 
\qquad \| \psi_p\upp{C} -  h(t) \|_2 \le \| \psi_p\upp{L} -  h(t) \|_2 . 
\en
In addition, 
\eq{eq:ineq2}
\| \psi_p\upp{C} -  h(t) \|_2^2 = \frac{\pi}{2} \sum_{j=p+1}^\infty \gamma_j^2 \le \frac{4 \pi }{3(p+1)} .  
\en
\end{proposition} 
\begin{proof}
The two 
inequalities in \eqref{eq:ineq1} follow from the optimality of the least-squares
approximation under the inner product \eqref{eq:ChebP}.
The first part of  \eqref{eq:ineq2}, i.e., the equality, follows from 
the fact that the scaled Chebyshev polynomials 
\[ \sqrt{ \frac{1+ \delta_{j0}}{\pi}} T_j(t) \] form a complete
orthonormal sequence of polynomials ($\delta_{ij}$ is the Kronecker
delta) with $T_j$ having squared norm equal to $\pi/2$ for
$j>0$.  For the second part, we have
\begin{eqnarray} 
\| \psi_p\upp{C} -  h(t) \|_2^2 = \frac{\pi}{2} \sum_{j=p+1}^\infty \gamma_j^2 
&=&  
\frac{\pi}{2} \sum_{j=p+1}^\infty 
\left[
\frac{2}{\pi}\left(
 \frac{\sin(j\arccos(a)) - \sin(j\arccos(b))}{j}
\right)\right]^2 \nonumber  \\
&\le &  
\frac{\pi}{2} \sum_{j=p+1}^\infty 
\left[
\frac{2}{\pi} \times 
 \frac{2}{j}
\right]^2 \nonumber  \\
& = &
\frac{8}{ \pi }
\sum_{j=p+1}^\infty 
\frac{1}{j^2}  . \label{eq:prop:1}
\end{eqnarray} 
Consider now the sum $\sum_{j=p+1}^\infty j^{-2}$. To simplify notation
we first start the sum from $p$ instead of $p+1$. Then we have:
\[
\sum_{j=p}^\infty \frac{1}{j^2} 
=
\sum_{l=1}^\infty 
\sum_{i=0}^{p-1} 
\frac{1}{(l p +i)^2} \\
\le
\sum_{l=1}^\infty 
\sum_{i=0}^{p-1} 
\frac{1}{(l p)^2} \\
= \sum_{l=1}^\infty 
p 
\frac{1}{(l p)^2} 
= \frac{1}{p} 
\sum_{l=1}^\infty 
\frac{1}{l^2} .
\] 
As is well-known $\sum_{l=1}^\infty l^{-2} = \pi^2/6$.
In the end we obtain: 
\[
\sum_{j=p+1}^\infty 
\frac{1}{j^2} \le \frac{1}{p+1} 
\sum_{l=1}^\infty 
\frac{1}{l^2} =  \frac{1}{p+1} \frac{\pi^2}{6} . 
\]
Substituting this into \eqref{eq:prop:1} yields 
 the inequality in \eqref{eq:ineq2}.
\end{proof}
 
As indicated by the inequalities \eqref{eq:ineq1}, we cannot do better
than the standard Chebyshev polynomial if we are to measure the
quality of the approximation by the $L_2$ norm based on the inner
product \eqref{eq:ChebP}. The convergence based on the bound
\eqref{eq:ineq2} is like $1/\sqrt{p}$, which is slow. However, much of
the inaccuracy occurs around the jumps of the function.  From a
practical point of view, the polynomial is available explicitly and it
is possible to analyze the error $ | \psi_p(t) - h(t) | $ in specific
subintervals to determine whether $\psi_p$ will be satisfactory. The
cost of this analysis is trivial since the degree is usually moderate.

\subsection{Generalized eigenvalue problem}
We now consider the generalized eigenvalue problem $Ax=\lambda Bx$ where
$A$ and $B$ are symmetric and $B$ is positive definite. In this case
the projector $P$ in \nref{eq:proj} becomes
\eq{eq:projG}
P = \sum_{\lambda_i \ \in \ [a  \ b]} u_i u_i^T B , 
\en 
and the eigenvalue count is again equal to its trace\footnote{Details on this can be found for example in \cite{Kramer-al-FEAST-2013}.}.
However, there are now two matrices involved and
 this projector does not admit an expression similar to that in
\eqref{eq:hoft} for  the standard case.
A common remedy to this issue is to compute the Cholesky factorization
$B=LL^T$ of $B$ and transform the generalized eigenproblem into a
standard one with the matrix $L^{-1} AL^{-T}$. This solution
reintroduces the need for a costly factorization which we wanted to
avoid in the first place. The following simple theorem yields the
basis for an efficient alternative:
\begin{proposition} 
  Let $B$ be a semi-positive definite matrix and $B = L L^T$ its
  Cholesky factorization.  Then the inertias of $A - \sigma B$ and $L
  \inv A L^{-T}- \sigma I$ are identical.
\end{proposition} 
\begin{proof}
It is well-known that the inertias of a matrix $C$ and 
$X C X^T $ are the same for any nonsingular matrix $X$, see,\
e.g., \cite{GVL-book}. The proposition follows by applying
 this result with $C = A - \sigma B$ and $X = L\inv$.
\end{proof}

A consequence of the above statement is that we can estimate
the number of eigenvalues of the pair $(A, B)$ located in a given
interval \emph{without resorting to any factorization}. 
In essence the idea is to convert the eigenvalue count for the pair 
$(A,B)$ into two eigenvalue counts for two standard eigenvalue
problems. Specifically, we have $\mu_{\ab} = \mu_a -\mu_b$ where
 $\mu_a$ is the number of positive 
eigenvalues of $A - a B$ and $\mu_b$ is
the number of positive eigenvalues of $A - b B $.

Thus, this approach requires a  Chebyshev expansion of the high-pass filters
\begin{eqnarray*}
f_\sigma(t) = \left\{
  \begin{array}{l l}
    1  & \quad \textrm{if}\ \ t \ \geq \ \sigma \\
    0  & \quad \textrm{otherwise}\\ 
  \end{array}
\right. 
\end{eqnarray*}
for $\sigma = a $ and $\sigma = b$ (see Fig.~\ref{fig:high-pass}, right,
 for an example). From these expansions, we would
get estimates for the desired counts $\mu_\sigma$ for $\sigma = a, b$, i.e., 
\[
f_\sigma(t)  \approx \sum_{j=0}^p 
\eta_j^{\sigma}  T_j(t) 
\qquad \rightarrow \qquad 
\mu_\sigma  \approx \sum_{j=0}^p 
\eta_j^{\sigma}  \trace [ T_j(A - \sigma B) ] .
\] 
Using the same set of sample vectors to estimate the two traces, we would then
get the following  eigenvalue count in the interval $\ab$:
\[
\mu_{\ab} = \mu_a - \mu_b 
 \approx \frac{n}{n_v} \sum_{k=1}^{n_v} 
\left[ \sum_{j=0}^p \eta_j^{a}  v_k^T T_j(A - a B) v_k  -
\sum_{j=0}^p \eta_j^{b} v_k^T T_j(A - b B) v_k \right]  . 
\]
As will be noted in the section devoted to the numerical experiments,
for truly generalized problems $(B \ne I)$, the spectrum 
distribution of the matrices $A-\sigma B$ for $\sigma = a, b$, may
lead to difficulties, requiring a very large degree polynomial in some
cases. It is also possible to count eigenvalues to the left of $a$ and
$b$ by using a low-pass filter. Notice that high-pass/low-pass filters
require usually a lower degree than mid-pass (`barrier') filters,  so
this can also be used for the standard eigenvalue problem not just the
generalized problem. As for the standard eigenvalue case, costly
factorizations are avoided at the expense of using two filters with
standard matrices.


\section{Rational expansion filtering} 
\label{sec:rational}
A natural extension to the idea of polynomial filtering
is to expand $P$ as a rational function. One of several ways
of achieving this expansion is via the Cauchy integral definition
of a projector: 
\eq{eq:cauchy}
P = - \frac{1}{2 i \pi } \int_{\Gamma} R(z) dz, 
\en 
 where $R(z) = (A-zI)\inv$ is the resolvent of $A$, $z \in
 \mathbb{C}$, and $\Gamma$ is some smooth curve in the complex plane
 containing the desired part of the spectrum (see, e.g.,
 \cite{Saad-book3}). Typically $\Gamma$ is taken to be a circle
 whose diameter is the line segment $\ab$.  The above integral is
 then approximated by resorting to numerical integration methods,
 leading to
\[
P \approx \chi_{n_c}(A) = \sum_{j=1}^{n_c}  \omega_j (A - z_j I)\inv ,  
\]
where the $z_j$'s are integration points and the $\omega_j$s are quadrature
weights. 
It is now possible to use the trace estimator by sampling with 
a set of random vectors. As shown in~(\ref{eq:expan}), the trace of $P$ will be 
approximated  by the average of $(Pv,v)$ over many sample
vectors $v$, multiplied by $n$
\eq{eq:estR}
\mu_{\ab} = \trace (P) \approx \frac{n}{n_v} \sum_{k=1}^{n_v} 
\left[
\sum_{j=1}^{n_c} \omega_j v_k^T (A - z_j I)\inv v_k \right] . 
\en
This method is not new as 
a similar idea was advocated in \cite{FutamuraSakuraiEigCnt}.

\subsection{Approximations for the spectral projector} 
\label{sec:ratGen} 
The rational expansion can be easily adapted to the case of a
generalized eigenvalue problem $Ax = \lambda B x $.  In this case the
desired projector is still given by Eq.~\nref{eq:cauchy}, but now
the resolvent becomes
\[
 R(z) = (A-z B)\inv B
\] 
(see~\cite{Kramer-al-FEAST-2013} for a simple derivation). This means
that the only change from the standard case is that the resolvent
$(A-z_j I)\inv$ in~(\ref{eq:estR}) must be replaced by $(A-z_j B)\inv
B$. In either case the stochastic estimation requires solving linear
systems with multiple right hand sides for each integration point. In
such situations it is customary to factorize the matrix $A-z_j B$
upfront so that the factors can be repeatedly used at a later
stage. In general such factors do not need to be calculated
exactly. For example, when employing an iterative procedure, an
approximate factorization of the matrix $A - z_j B$ can be used as a
preconditioner for the linear solver.

From the computational cost point of view, this approach may appear to be
expensive and not competitive with the one based on Sylvester's
inertia theorem described in Sec.~\ref{sec:eigcount}. Indeed, the
inertia approach requires only two factorizations whereas we may now
need a few such factorizations to get a good approximation to the
spectral projector. In reality the method based on Sylvester's
inertia must utilize an exact factorization, whereas in the above
formula, all that is needed is to solve linear systems $(A-z_j I) y_k
= v_k$ for many right-hand sides $v_k$ by any inexpensive procedure,
including an iterative one. 
 
The rational expansion approach should be favored when used in
combination with the FEAST eigensolver
\cite{FEAST,FEASTsolver,FEASTdoc} (or similar methods). Such an
expansion allows us to get a rough eigenvalue count when the
factorizations of $A - z_i B$ have been already computed in
preparation for a subspace iteration-like procedure used by FEAST for
the symmetric problem. In such procedure Gauss-quadrature points
(i.e. shifts $z_i$) are positioned along the half-circle contour as
shown on the left of Fig.~\ref{fig:gauss}. The same quadrature
points can be used to compute the rational approximation for the
projector $\chi_{n_c}$ which is plotted, as a function of $\lambda$,
on the right side of Fig.~\ref{fig:gauss}.
\begin{figure}[hbt]
\centering
  \begin{subfigure}[b]{0.495\textwidth}
                \centering
                \includegraphics[width=\textwidth]{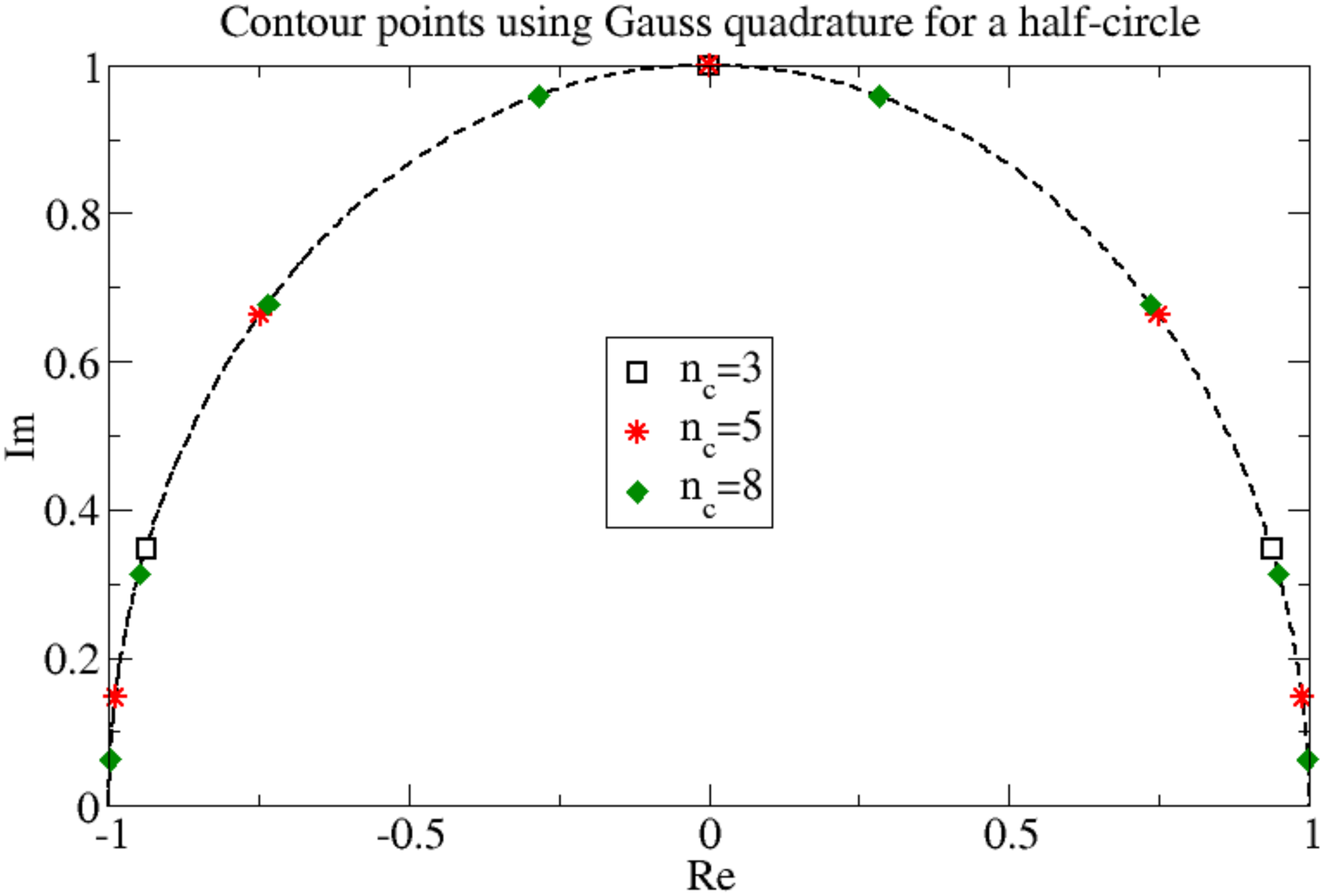}
                \caption{Quadrature points.}
                \label{fig:quadnod}
  \end{subfigure}
  \begin{subfigure}[b]{0.495\textwidth}
                \centering
                \includegraphics[width=\textwidth]{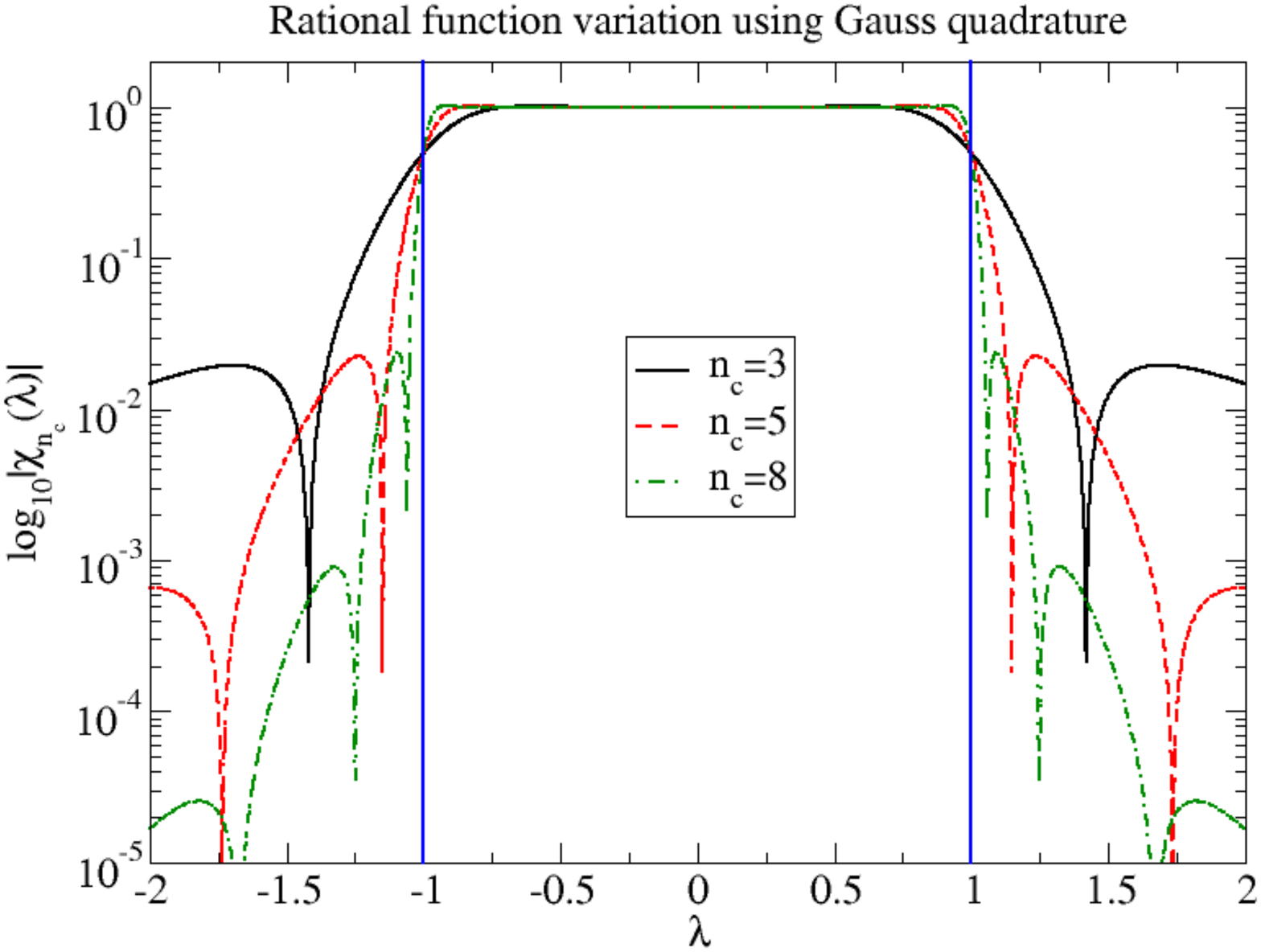}
                \caption{Rational approximation filter.}
                \label{fig:rafil}
  \end{subfigure}
  \caption{Position in the complex plane of the $n_c$ Gauss quadrature
    points using an integration along the half-circle, and the
    corresponding values of the rational function
    $\chi_{n_c}(\lambda)$ using a semi-log plot. The search interval
    is here set to $[-1, \ 1]$.}
\label{fig:gauss}
\end{figure} 
Note the rapid decay of the rational function from $\simeq 1$ in the
middle of the interval $[-1, \ 1]$ to $\simeq 0$ outside, which both
explain the expected efficiency of the trace estimator
(\ref{eq:estR}), and can lead to some remarkable convergence rates for
the FEAST subspace iteration procedure.  More precisely, for the
eigenpair $(\lambda_j, u_j)$, the error introduced by the numerical
integration on the projector is bounded by (Theorem 4.1 in
\cite{Tang13}):
\begin{equation}
\|(h - \chi_{n_c})u_j\|_B \leq \alpha\left|\frac{\chi_{n_c}(\lambda_{M_0+1})}{\chi_{n_c}(\lambda_{j})}\right| \quad \quad \forall j\in 
1,\dots,\mu_{\ab} 
\label{eq:feast_ratio}
\end{equation}
where the first $\mu_{\ab}$ eigenvalues $\lambda$ are the ones
inside the search interval $[-1, \ 1]$, and $M_0$ is the size of the
search space. The ratio between the rational functions in
(\ref{eq:feast_ratio}) is inversely proportional to the FEAST
convergence rate, and it is expected to decrease exponentially with
the number of integration points $n_c$. In turn, the value of this ratio
may increase significantly in case one or both of the search interval
boundaries is (are) near clusters of the spectrum (i.e. if $M_0$ is
too close to $\mu_{\ab}$ for FEAST). Similarly, the eigenvalue
count may not be very accurate in this case (see
Sec.~\ref{sec:biasclst} for more details).

\subsection{Practical considerations}
\label{sec:cons} 
Next we consider a few implementation issues related to the rational
approximation filtering approach.  The eigenvalue count estimator is
based on the formula \eq{eq:est} \trace (P ) \approx \frac{n}{n_v}
\sum_{k=1}^{n_v} \sum_{j=1}^{n_c} \omega_j v_k^T (A - \sigma_j I)\inv
v_k .  \en The above formula involves two loops: the $k$-loop which we
will refer to as the `sample vector loop' and a $j$ loop which we call
the `integration loop' with $n_c$ the number of integration points.
As it is written, the above formula suggests that we would run a
vector loop, in which we would generate random vectors, then for each
vector in turn we would solve $j$ right-hand sides (integration loop).
This is fine when a direct solver is used for the solutions, provided
we store the factorizations for each integration point.
 
An important observation here is that we can also swap the two loops,
in effect exploiting the fact that the trace of the sum of operators
is the sum of the different traces:
\begin{eqnarray*} 
  \trace (P ) &\approx& \trace \sum_{j=1}^{n_c} \omega_j (A - \sigma_j I)\inv \\
  &=& \sum_{j=1}^{n_c}  \omega_j \trace (A - \sigma_j I)\inv \\
  &\approx& \frac{n}{n_v} 
  \sum_{j=1}^{n_c} \omega_j \sum_{k=1}^{n_v} v_k^T (A - \sigma_j I)\inv v_k
  . \end{eqnarray*}
This can be quite useful in a processing phase: As each of the 
$n_c$ factorizations is obtained we generate a number of random vectors
and estimate the trace of $(A-\sigma_j I)\inv$ with them. At the end of all
the factorizations, we end up with an eigenvalue count estimate 
which can be exploited to determine the subspace dimension to use
in the FEAST subspace iteration procedure (i.e. value of $M_0$ in (\ref{eq:feast_ratio})).

Iterative solvers offer an appealing alternative to exact factorizations.
When using a Krylov subspace method without preconditioning,
one can immediately make the well-known observation that the
various systems $(A-\sigma_j I) y_j  = v$ for different $j$'s 
and for each random  vector $v$, can all be solved with the same Krylov
subspace (e.g.~\cite{LopezSimoncini06}).  

Another issue is to determine what accuracy to require from the
solver. In the context of the FEAST eigensolver  \cite{FEASTsolver}, 
the residual norm
criterion for GMRES will clearly yield a similar size error for the
eigenvector as it was observed and analyzed in \cite{Kramer-al-FEAST-2013,Tang13}. The
problem for counting eigenvalues is slightly different. A first
observation is that we are not interested in computing eigenvalues or
eigenvectors.  It is well-known that the error made on the eigenvalues is
typically of the order of the square of the related residual
norm~\cite{Saad-book3}. From our observations, a high accuracy is not
needed. However, there is a minimum accuracy required, below which the
method will no longer work. It is also important to have consistent
error thresholds. For example, just using a fixed number of GMRES
steps will usually not work. It is best to use a criterion based on a
residual norm reduction, for example by a factor of $10^{-2}$.

\section{Numerical experiments}
\label{sec:numtests}
This section provides numerical illustrations of a 
number of features and discusses additional issues of the
methods proposed in Sec.~\ref{sec:polynomial} and
\ref{sec:rational}.

\subsection{Polynomial filtering: Standard and Jackson polynomials}
\label{sec:jackpoly}

\begin{figure}[hbt]
\begin{center} 
\includegraphics[width=0.495\textwidth,height=0.4\textwidth]{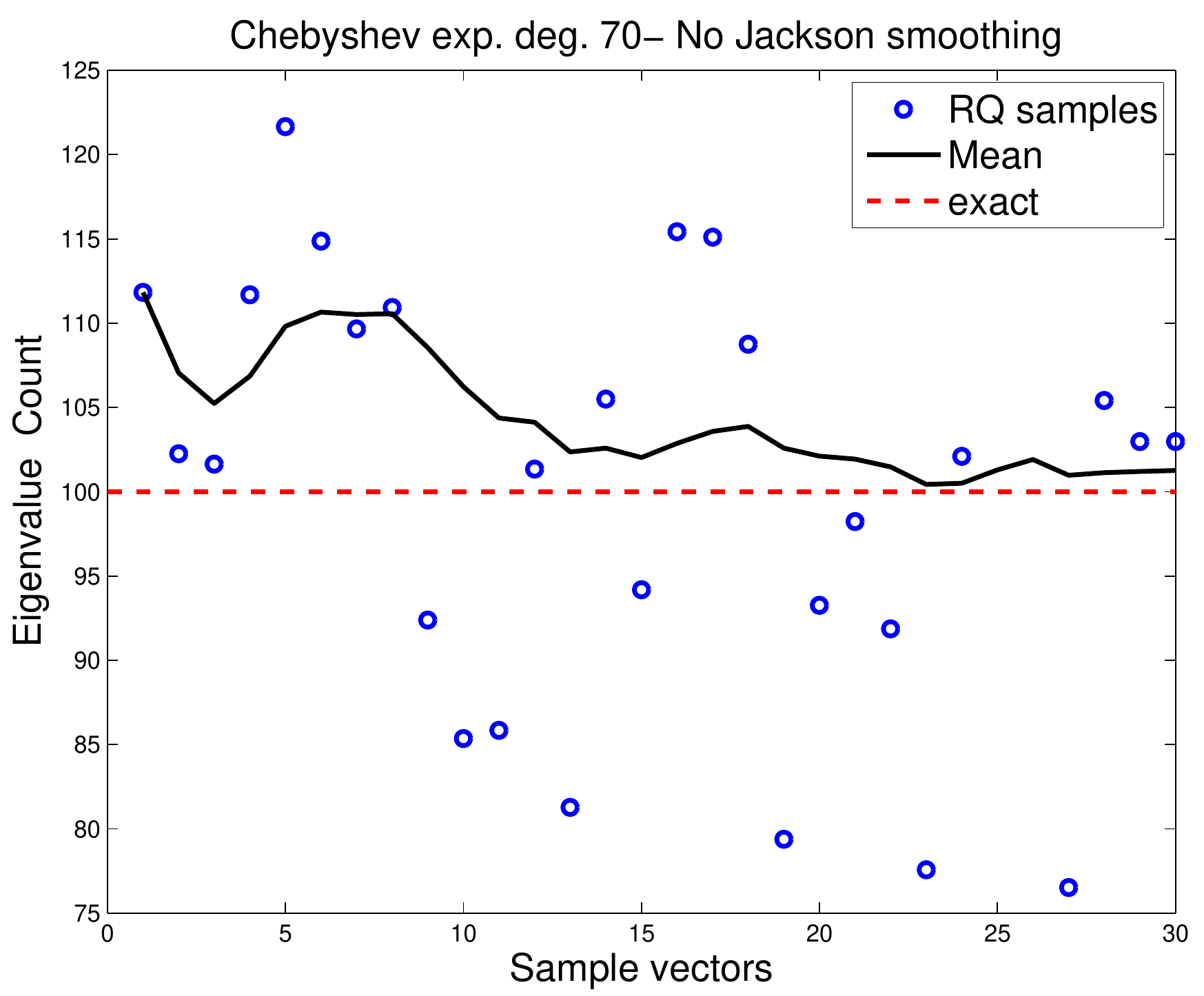} 
\includegraphics[width=0.495\textwidth,height=0.4\textwidth]{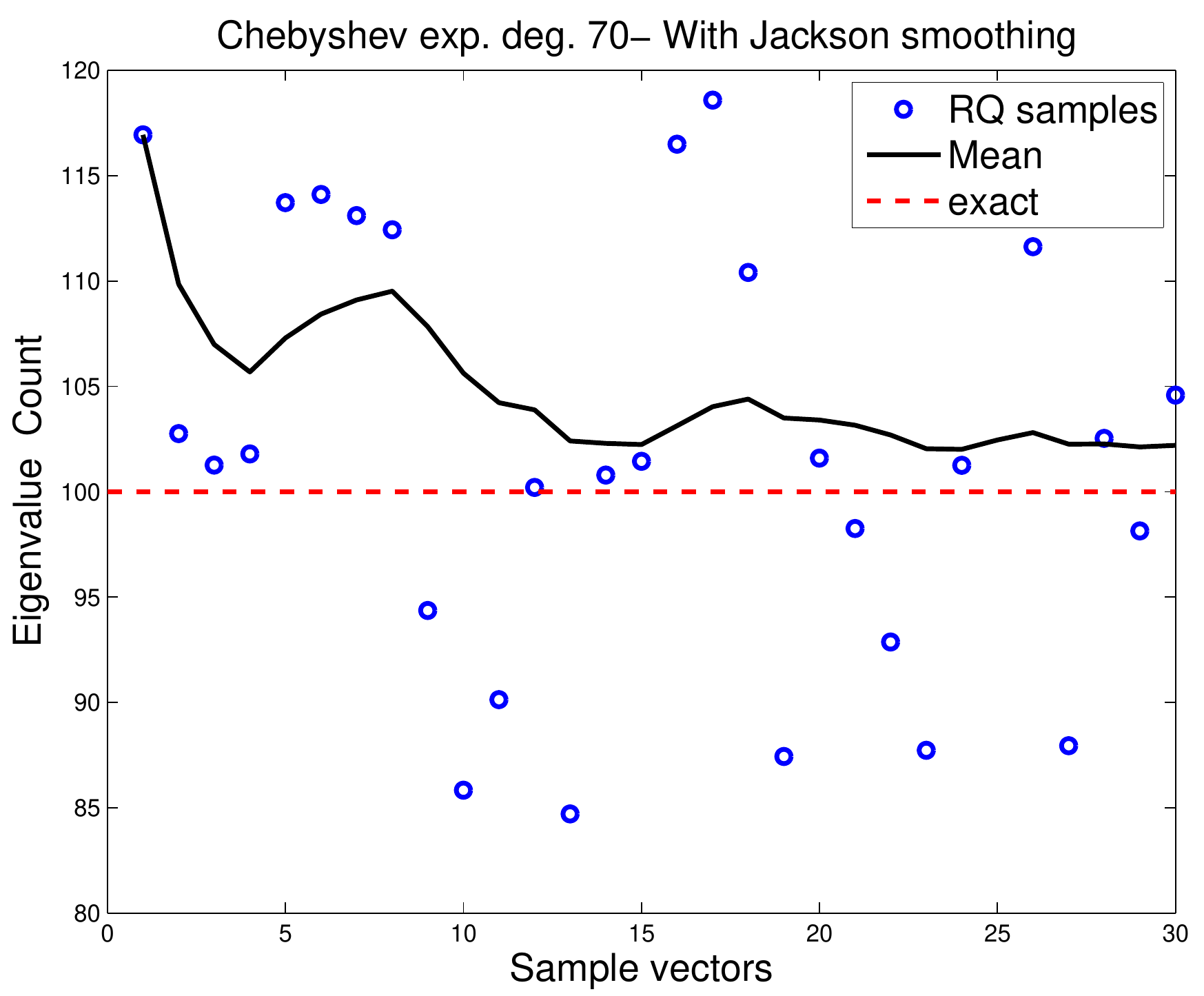} 
\end{center}
\caption{Chebyshev vs Jackson-Chebyshev for counting eigenvalues
$\lambda_{101} $ to $\lambda_{200} $  for the matrix Na5.
\label{fig:chebNa5}}
\end{figure}

We use as first example the Na5 matrix generated by the PARSEC
code. This matrix, available from the University of Florida matrix
collection
\footnote{\url{http://www.cise.ufl.edu/research/sparse/matrices/}}, is
of size $n = 5832$ and has $ nnz = 305630$ nonzero entries.  We
computed eigenvalues at the outset and defined the interval $\ab$ so
that $a$ is in the middle of $\lambda_{100}, \lambda_{101}, $ and $b$
is in the middle of $\lambda_{200}, \lambda_{201}$. In this situation
the exact eigenvalue count is 100. Using $n_v = 30$ and a degree 70
standard Chebyshev polynomial yields the results shown on the left
plot of Fig.~\ref{fig:chebNa5}. The right plot of the same figure
shows the result obtained with Jackson-Chebyshev polynomials of the
same degree: the last value of the computed average (eigenvalue count
estimate) was 101.25 for Chebyshev and 102.20 for
Jackson-Chebyshev. The same sequence of random vectors were used in
both Chebyshev and Jackson-Chebyshev.

From this specific example, one may conclude that Jackson tends to
often give an overestimate whereas standard Chebyshev often gives an
underestimate. In reality, the behavior of both methods depends
crucially on the eigenvalue density distribution in relation to the
position of the extrema of the interval $\ab$ and the situation is
often reversed (see next subsection).

\subsection{Estimate bias}\label{sec:bias}  
In this section we consider only polynomial methods, although similar
statements can be made for the rational approximation methods.  Since
$\psi_p(A)$ is only an approximation to the projector \nref{eq:proj},
its trace will not be equal to the number of eigenvalues inside the
interval. In some situations (not involving clustering) a relatively
large degree is needed to get a reasonable approximation. An
illustration of what can happen when the degree is not large enough is
shown in Fig.~\ref{fig:Si2Bias}.  On the left side the lower
horizontal line is the trace of the matrix $\psi_p(A)$ computed with a
low degree. The higher dashed horizontal line is the actual eigenvalue
count. There is a substantial gap between the two showing that the trace
estimator does work, but it estimates a trace of an inaccurate
projector. The right side of the figure shows that the gap narrows
substantially for a higher degree polynomials. As shown in
Table~\ref{tab:bias}, a degree above 70 is necessary to get an
approximation that is close enough, where the lower and upper dashed
lines are close. In this regard there could be a big difference
between the Jackson-Chebyshev and the standard Chebyshev
polynomials. Here Jackson smoothing seems to be very detrimental to
the estimation. In other situations, the Jackson polynomial performs
better.

\begin{figure}[hbt]
\begin{center} 
\includegraphics[width=0.495\textwidth,height=0.4\textwidth]{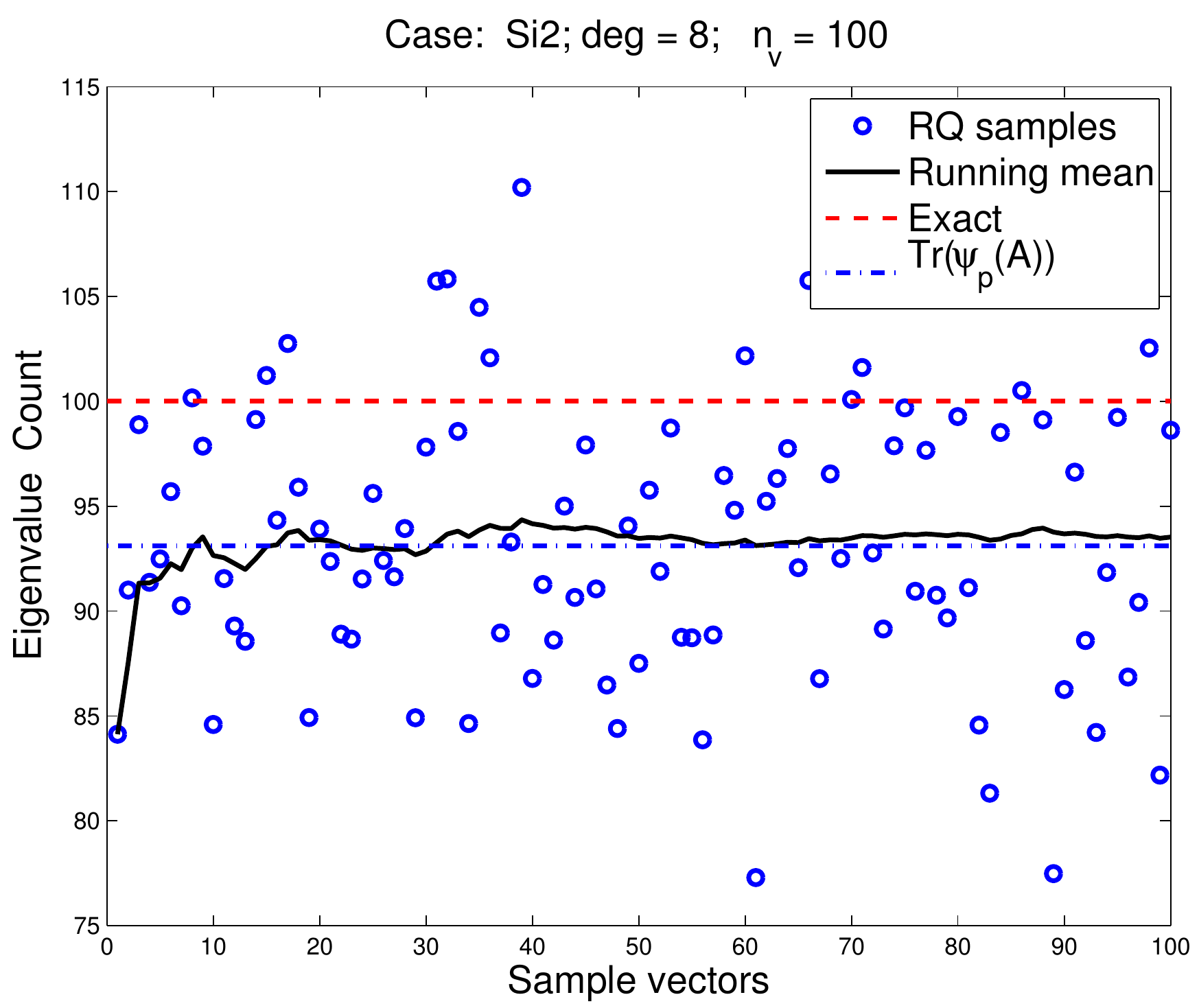} 
\includegraphics[width=0.495\textwidth,height=0.4\textwidth]{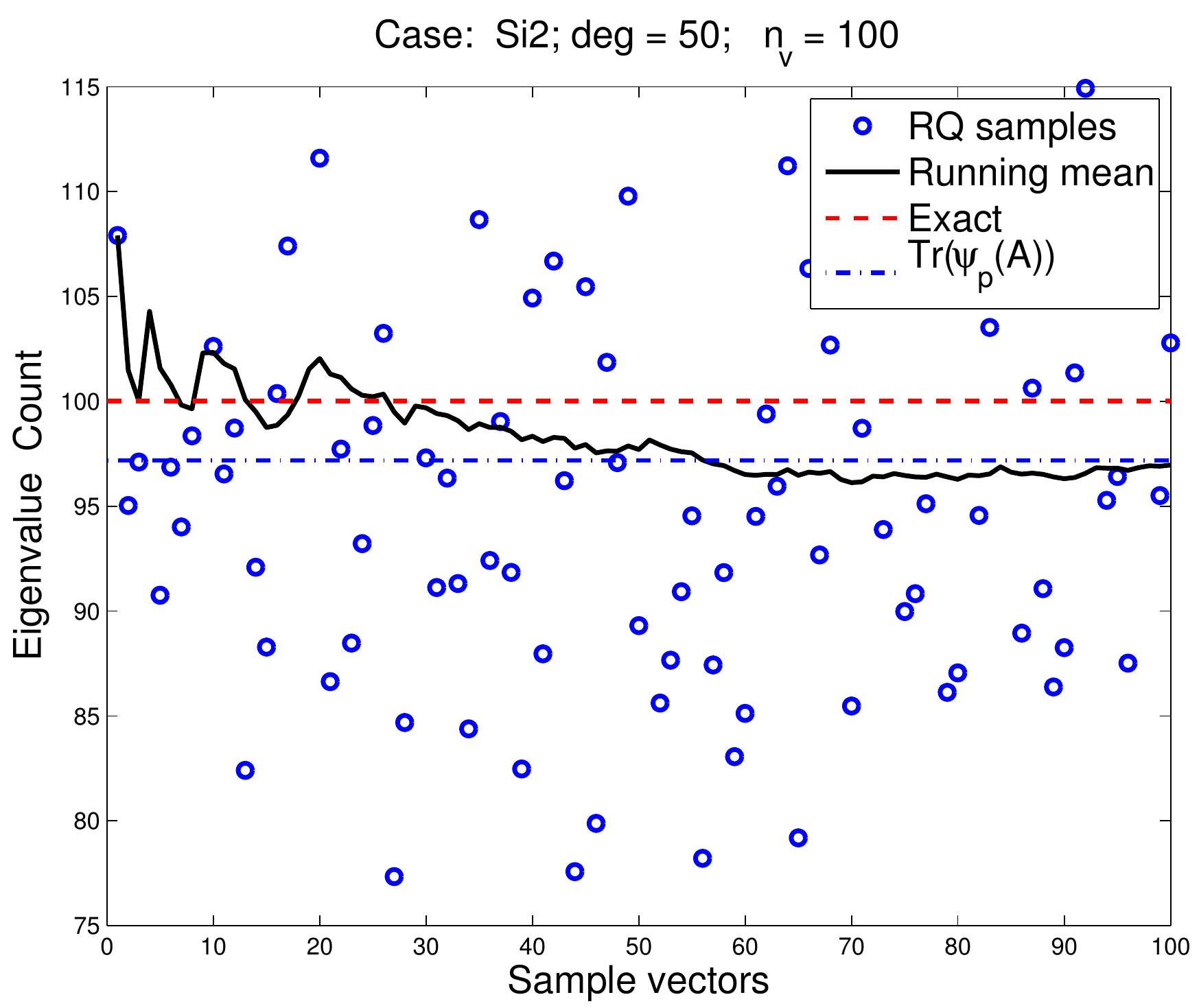} 
\end{center} 
\caption{Chebyshev based counting with a degree 8 (left) and
50 (right) for matrix Si2.
\label{fig:Si2Bias}}
\end{figure}
  
We explored a few ways to fix this bias. In particular the simplest
correction is to compare the integrals of $\psi_p(t)$ in $[-1, \ 1]$
with that of the step function on the same interval. The integral of
the step function is just $b-a$.  The integral of each $T_k$ in the
interval $[-1, \ 1]$ is readily computable (it is equal to zero when
$k$ is odd and to $-2/(k^2 -1)$ when $k$ is even).  Then one can
obtain the integral of $\psi_p$ in $\ab$ from which a corrective
factor can be obtained. However, experiments with such a correction
were mixed. The difficulty inherent to this problem is that the bias will
certainly depend on the distribution of eigenvalues.

\begin{table}[htb]
\begin{center}
\begin{tabular}{|l|r|r|r|r|r|r|r|r|} \hline 
  $p$  &  8 &  20 &  30 &  40 &  50 &  70 & 100 & 120 \\ \hline
 tr $ [ \psi_p^{(C)} (A) ] $ &
  93.12 &  97.29 &  96.98 &  96.81 &  97.19 & 101.58 & 101.54 & 100.76\\  \hline
tr  $ [ \psi_p^{(J)}(A)] $ & 74.53 & 89.59 & 93.00 &
 94.51 & 95.29 & 95.97 & 96.99 & 97.74 \\ \hline 
\end{tabular}
\end{center} 
\caption{Evolution of the trace of $\psi_p(A)$
  for the  standard Chebyshev $\psi_p^{(C)}$ and
  Jackson-Chebyshev $\psi_p^{(J)}$
  approaches for the Si2 test case. The exact count is 100.
  \label{tab:bias}}
\end{table}

\subsection{Eigenvalue clustering}
\label{sec:biasclst}
As was mentioned at the end of Sec.~\ref{sec:jackpoly}, the
accuracy of the eigenvalue estimate depends on the eigenvalue
distribution of $A$. In particular, if either $a$ or $b$ is close or
inside a cluster of eigenvalues, $\mu_{\ab}$ may overestimate or
underestimate the true number of eigenvalues in $\ab$. Moreover the
extent of the error is related to both the size of the cluster and the
relative distance $D_\lambda=\frac{\|\lambda_{i+1} -
  \lambda_i\|}{\|\lambda_i\|}$ between eigenvalues in the cluster.

In order to verify this point we ran a series of
tests on an eigenproblem containing artificially engineered
clusters of distinct length and density.  For each cluster we placed
the rightmost ending of the interval $\ab$ either at the beginning
or at the end of the set of values forming the cluster. For each case
we computed $\mu_{\ab}$ using both the simple Chebyshev
$\psi^{(C)}_p(A)$ and the Jackson-Chebyshev expansion $\psi_p^{(J)}(A)$ with
two quite different polynomial degrees $p$.

Results from our tests clearly show that when the filtering interval
intersects the beginning of a cluster both $\psi^{(C)}_p(A)$ and
$\psi_p^{(J)}(A)$ overestimate the eigenvalue count. In contrast when the
interval intersects the end of a cluster both expansions underestimate
the eigenvalue count. This result does not in general depend on the
polynomial degree adopted and reflects the tailing effects of the
expansions as shown for example in Fig.~\ref{fig:mid-pass}. Moreover
we observed that the extent of the error made depends on the length of
the cluster and on the relative distance between eigenvalues in the
cluster.

We repeated the set of tests for the rational expansion and observed a
behavior similar to polynomial expansions. When most of the values
characterizing a cluster are close but outside the interval $\ab$,
the trace of $\chi_{n_c}(A)$ overestimates the number of
eigenvalues. In contrast, if such cluster of values is almost entirely
within the filtering interval, the eigenvalue count is underestimated.
In other words, clusters located near the end-points will just
exacerbate the bias described in section~\ref{sec:bias}.


\subsection{Number of sample vectors and estimator accuracy}
\label{sec:vectrace}
According to lemma \ref{th:toledo} there exist a minimum value
of vectors for which we can assert we converged to the value of the
trace with probability $1-\delta$. While this bound is rather loose it
indicates that convergence is reached for $n_v \sim
\mathcal{O}({\trace(P)})$ for any meaningful $\delta$. In this section we
illustrate the numerical behavior of the estimator with two typical
examples which we briefly comment.
\begin{figure}[!hbt]
\centering
  \begin{subfigure}[b]{0.495\textwidth}
                \centering
                \includegraphics[width=\textwidth,height=0.808\textwidth]{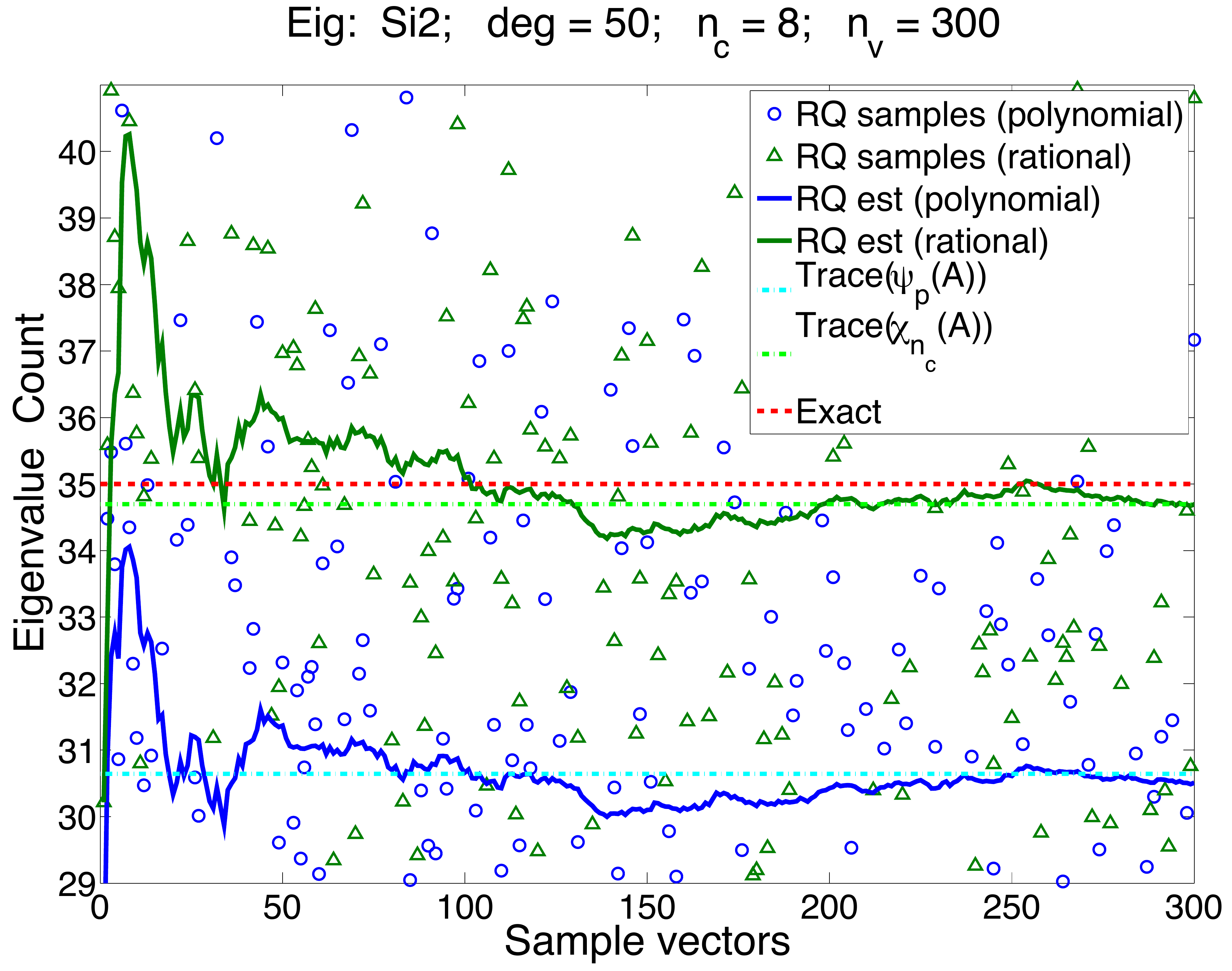}
                \caption{Small interval $\ab$ including $\lambda_5$ 
                    up to $\lambda_{40}$.}
                \label{fig:intrv1}
  \end{subfigure}
  \begin{subfigure}[b]{0.495\textwidth}
                \centering
                \includegraphics[width=\textwidth,height=0.808\textwidth]{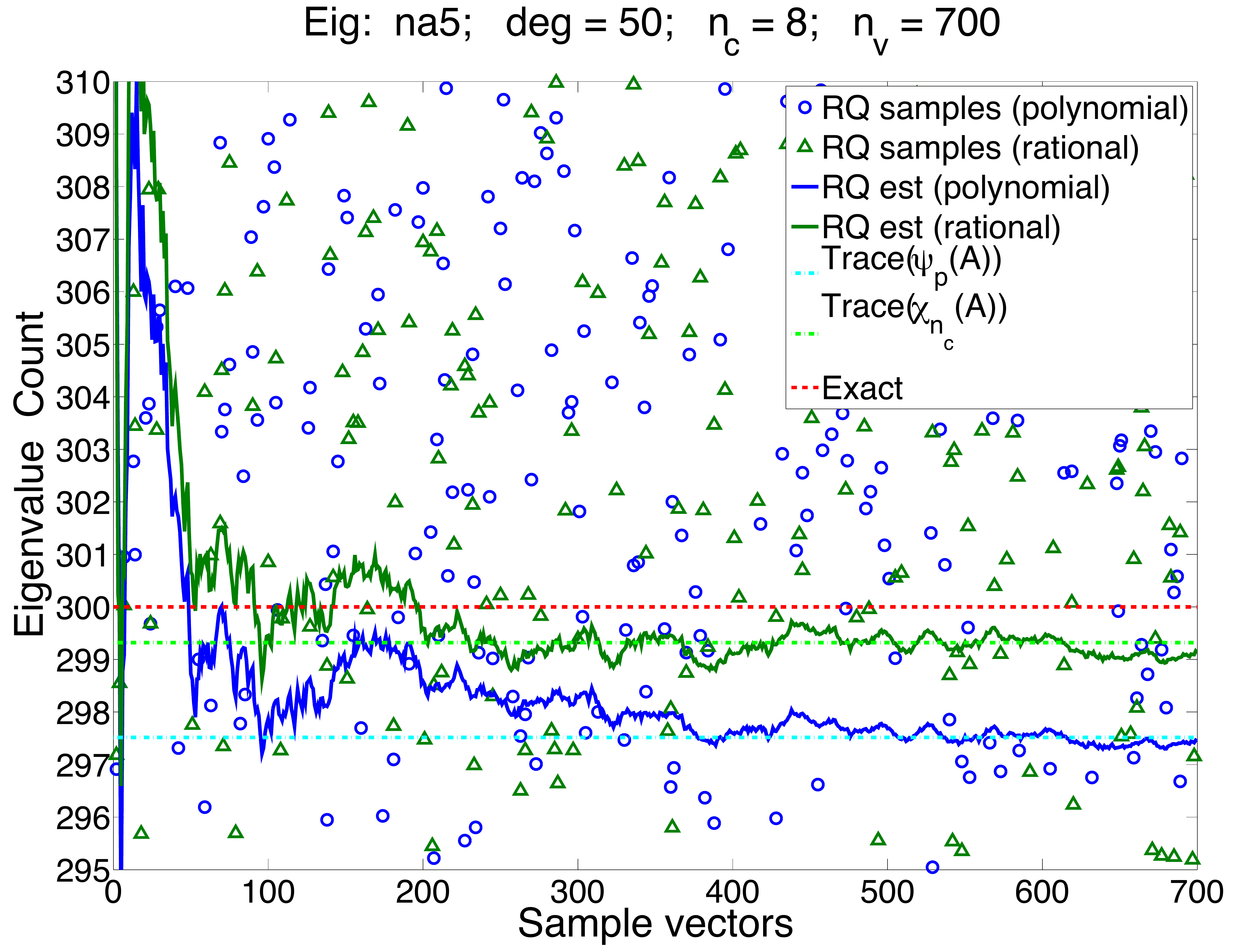}
                \caption{Large interval $\ab$ including $\lambda_{100}$ up to
                    $\lambda_{400}$.}
                \label{fig:intrv2}
  \end{subfigure}
  \caption{Behavior the Rayleigh Quotient estimator when
     the number of sample vectors $n_v$ varies. Two distinct test cases
 are considered with 
 both the polynomial and rational expansion techniques.
 (a):  Case of a small problem
 ($n=789$) and an interval containing a small number of
  eigenvalues.   (b): Both the system size and number of
  eigenvectors are one order of magnitude larger. In each plot the
  random generated vectors are the same for both estimators and the
  straight color lines are the exact traces of the approximate
  projector.}
  \label{fig:estvsNv}
\end{figure}

In both plots of Fig.~\ref{fig:estvsNv} the jagged lines representing
the estimates quickly reduce their wide oscillations and remain close
to the exact trace of the approximated $P$. The value of $n_v$ for
which the range of the oscillations from one sample vector to the next
become consistently less than one is between $50 \leq n_v \leq 100$
for Fig.~\ref{fig:intrv1} and $200 \leq n_v \leq 300$ for
Fig.~\ref{fig:intrv2}. These two typical examples illustrate that, for
intervals containing a small number of eigenvalues, a minimum number
of vectors $n_v \sim \trace(P)$ may be needed. In fact a good estimate
for the trace is based on having enough sample vectors $v_k$ to
represent almost entirely the subspace generated by the eigenpairs in
the interval $\ab$ and such a number is exactly $\trace(P)$.

When the number of eigenvalues in $\ab$ is larger than a few dozens,
the necessary number of sample vectors can often be lower $n_v
\lesssim \trace(P)$.  In general, for the same approximated projector,
the minimum value of $n_v$ varies slightly every time a different set
of sample vectors is selected, but overall the experimental bound
differs from the theoretical one given by lemma \ref{th:toledo} having
a much smaller value for the pre-factor of $\trace(P)$.

While this conclusion seems natural due to the nature of the
expansions, it may introduce a practical difficulty in estimating a
priori the minimum number of sample vectors necessary for the
estimation of the trace. This difficulty can be overcome by computing
the estimator incrementally and monitoring the value of the increment
over a small range of previous samples (typically no more than 10).
When such an increment remains substantially constant and less than
one, the estimated value for the trace can be safely said to have
converged.

\subsection{Rational approximation filtering: Direct vs iterative solvers}
\label{sec:dvsi}
We ran tests with the same Na5 example which was used in
Sec.~\ref{sec:jackpoly}.  When employing exact factorizations for each
integration point and $n_v=40$ sample vectors, the eigenvalue count is
$\mu_{\ab} = 98.64$ for $n_c = 3$ and $\mu_{\ab} = 100.27$ for
$n_c = 5$. Clearly more accurate quadrature rules will yield
more accurate counts.  Note however, that much faster results can be
obtained using the Sylvester inertia approach since this requires
only two direct factorizations (in real arithmetic).

We can avoid costly direct factorizations and make use of iterative
solvers to improve efficiency. For the Na5 example with same
parameters as above, we solved the linear systems using GMRES with
residual norm equal to $10^{-1}$, $10^{-2}$ and $10^{-3}$ and obtained
a trace estimate equal respectively to 75.35, 97.61 and 98.65.  As
mentioned in Sec.~\ref{sec:cons} a minimum accuracy for the iterative
solver is necessary but we observe that the estimates are reasonably
good starting with a high value for the residual threshold
(e.g. $10^{-2}$ for GMRES) when compared with the direct case. Using
$n_c=3$, Fig.~\ref{fig:errordirectgmres} shows the absolute errors on
$\mu_{\ab}$ obtained by comparing direct factorizations and GMRES with
$10^{-2}$ and $10^{-3}$ residual threshold. Results based on the
iterative solver are found in extremely good agreement with those of
the direct factorization, and this agreement is expected to be
preserved independently of the number of integration points.

When there is no preconditioning, the iterative procedure requires
only matrix-vector multiplications with the matrix $A$. Therefore,
using an iterative solver is a promising method for handling very
large systems. In order to further enhance the performance, one can
think of two options to be developed: (i) the generation of a single
Krylov subspace common to all integration points; (ii) the use of
(cheap) preconditioners.

\begin{figure}[hbt]
\centerline{ 
\includegraphics[width=0.5\textwidth,height=0.4\textwidth]{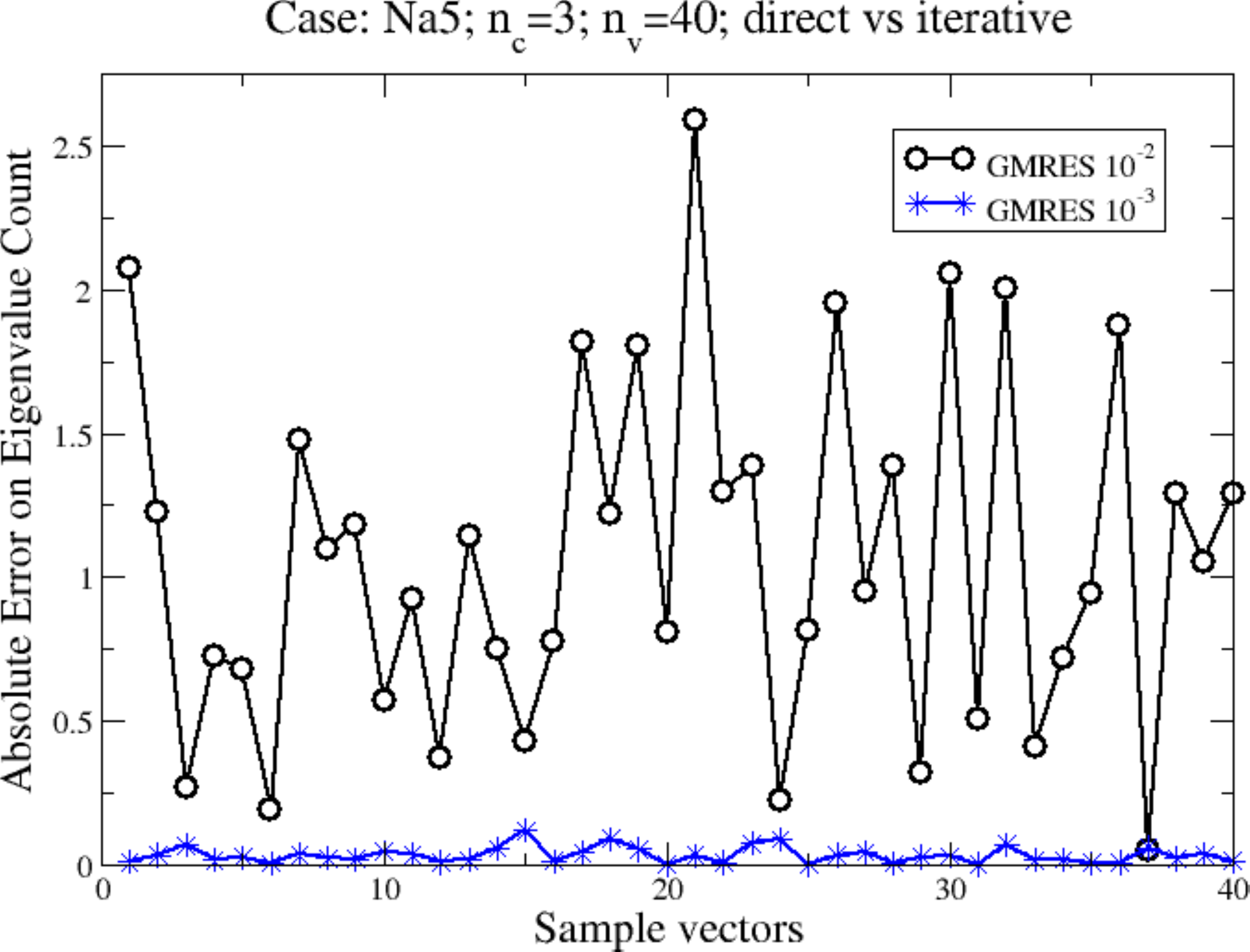}} 
\caption{Absolute 
errors on the eigenvalue count 
between results obtained with direct factorization and 
GMRES using both $10^{-2}$ and $10^{-3}$ convergence criteria. 
This absolute error (i.e. $|\mu_{\ab}^{Direct}-\mu_{\ab}^{GMRES}|$) decreases to zero
using smaller GMRES residual thresholds.
\label{fig:errordirectgmres}}
\end{figure}

We also tested the polynomial and rational filtering methods on a
generalized eigenvalue problem corresponding to a 2-D FEM simulation
matrix and available from the FEAST package \cite{FEASTsolver}.  In
this case the matrices $A, B$ have size $n = 12,450$. The number of
nonzero entries in both $A$ and $B$ is $nnz = 86,808$.  The number of
eigenvalues inside the desired interval is 100 (i.e. 100 lowest
eigenvalues).  On the left plot of Fig.~\ref{fig:Gen2D} the varying
average of $\mu_{\ab}$, obtained with a degree 100 polynomial
filtering, shows an almost perfect agreement with the exact count.

For the rational filtering method we run tests with $n_c = 5$, using a
direct solver and GMRES with and without preconditioning (denoted by
GMRES and P-GMRES respectively). Comparison of these approaches are
shown on the right plot of Fig.~\ref{fig:Gen2D}.  For clarity the
figure now omits the small circles corresponding to Rayleigh quotients
of each sample and shows only the running mean. We tested the
preconditioned GMRES with an ILU factorization using a drop tolerance
of $\texttt{droptol} = 0.01$ and a pivoting threshold of 0.05 as
defined by the \MATLAB\ \texttt{ilu} function with its default
reordering. For all the runs with P-GMRES we used a restart dimension
of 20 and limited the number of steps to 200.  The iteration is
stopped when the residual norm drops by a preset threshold
tolerance (\texttt{tol}).  Results with $\texttt{tol} = 10^{-2}$ are
shown.  Other results obtained with GMRES and P-GMRES with
$\texttt{tol} = 10^{-3}$ were quasi-identical with those obtained with
the direct solver and are omitted.

The curves show an underestimation of the eigenvalue counts for both
GMRES and P-GMRES.  Remarkably, each of these two curves deviates from
that of the exact solver by nearly a constant, indicating the
existence of a bias for the case of insufficient accuracy. This
observation is not easy to explain and merits further investigation.

The better accuracy afforded by the rational approximation approach
comes at a high memory cost.  The total number of nonzero elements for
all the LU factorizations generated for the 5 integration points for
this case is $nnz_{LU total} = 22,239,530$, an enormous amount
relative to the original number of nonzero entries which is under
$100,000$. In contrast, when the ILU factorization is used for $n_c =
5$ and under the conditions of the above experiment the total number
of nonzero elements used for all 5 ILU factorizations drops to
$nnz_{LU total} = 799,143$, about 28 times smaller. This is a 2-D
problem. For a 3-D problem a direct approach becomes unfeasible on a
standard workstation and using a preconditioned iterative solver is
the only option.

\begin{figure}[hbt]
\begin{center}
 \includegraphics[width=0.495\textwidth,height=0.4\textwidth]{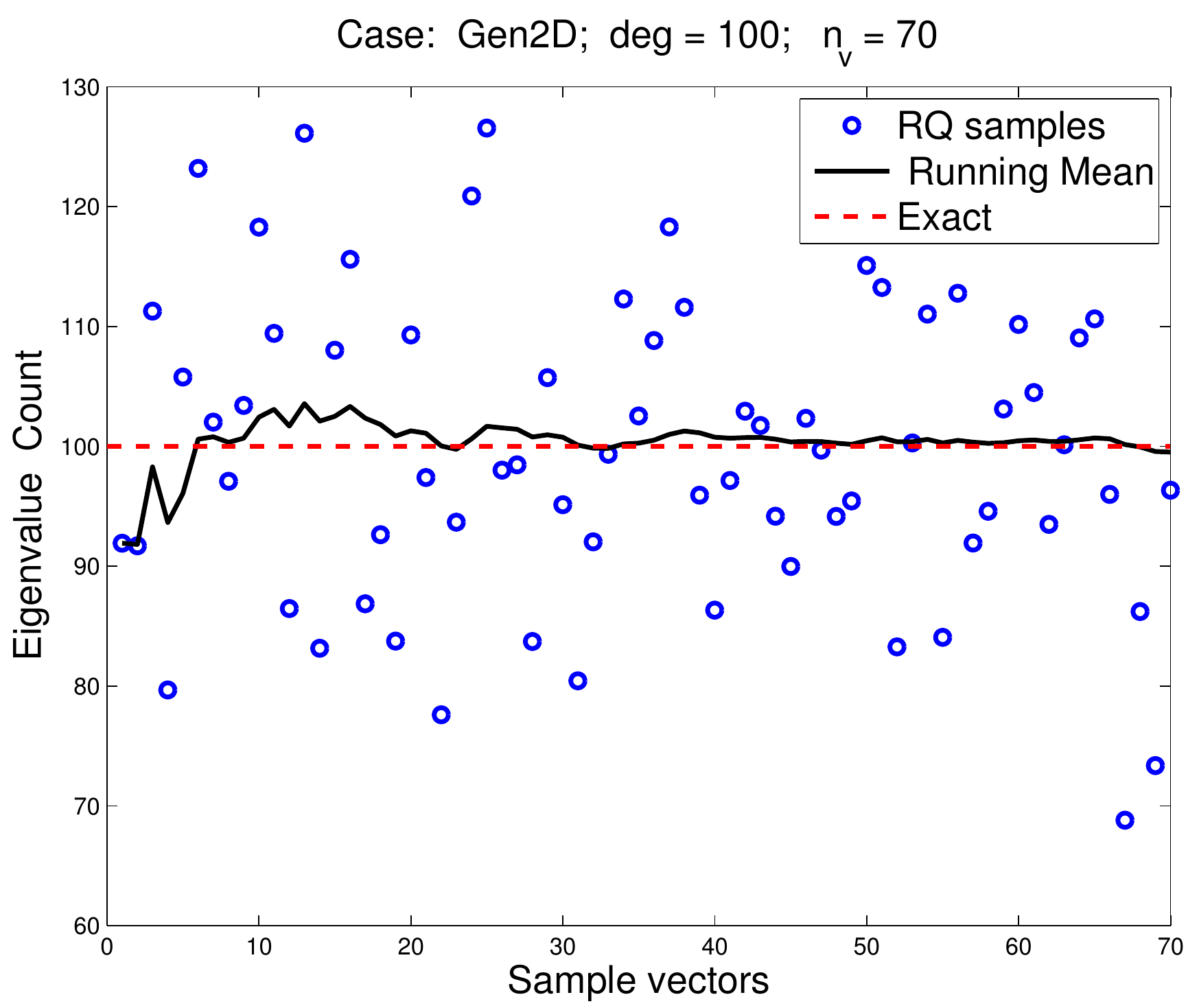} 
 \includegraphics[width=0.495\textwidth,height=0.4\textwidth]{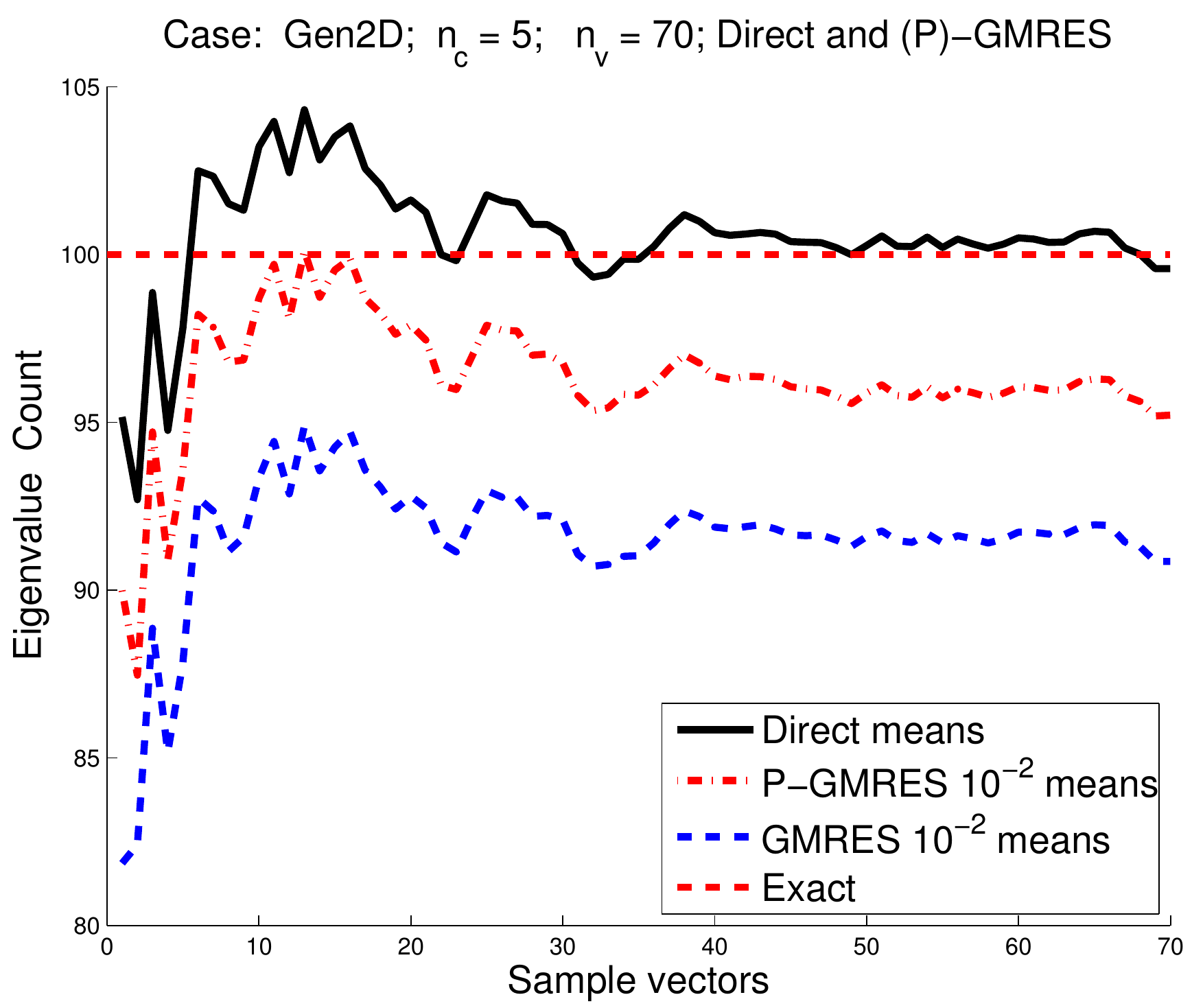}
\end{center} 
 \caption{(Standard) Chebyshev vs rational approximation for 
 the Gen2D matrix. The sequences of 
 random vectors used in all the experiments are identical.
 \label{fig:Gen2D}}
 \end{figure}

\subsection{Using estimated eigenvalue counts in  FEAST}
As mentioned in Sec.~\ref{sec:ratGen}, the convergence rate of the
FEAST subspace iterations mainly depends on the size $M_0$ of the
search subspace and the value of the rational function at
$\lambda_{M_0+1}$ (i.e.  $\chi_{n_c}(\lambda_{M_0+1})$ in
(\ref{eq:feast_ratio})).
In Fig.~\ref{fig:gauss}, for example, we note for the case $n_c = 8$
and the search interval $\lambda\in[-1, \ 1]$ that the rational
function is equal to $\sim 10^{-4}$ for $\lambda = \pm1.5$.  As a
result, if the search subspace size $M_0$ is taken large enough to
include all eigenvalues between $[-1.5, \ 1.5]$, one can expect the
residuals of the eigenpairs within $[-1, \ 1]$ to converge with the
same (linear) rate of $10^{4}$ along the FEAST subspace iterations
(while the eigenvalues should converge with a linear rate of
$(10^{4})^2=10^{8}$).  

It is important to note that the same convergence rate applies to any
arbitrary intervals $\ab$ if one chooses $M_0$ equal to the eigenvalue
count inside a larger interval $[a-\alpha, \ b+\alpha]$ with
$\alpha=(b-a)/4$. This eigenvalue count within the larger interval
can be estimated, in turn, using a reduced number of integration
points that can take advantage of iterative solvers with modest
residuals (as discussed in Sec.~\ref{sec:dvsi}).  The results obtained
in Table \ref{tab:feastruns} for the Na5 example, illustrate how an
appropriate estimate on a larger interval can be used by FEAST to
guarantee a certain degree of convergence rate for the interior
eigenpairs.
\begin{table}[htbp]
\begin{center}
\begin{small}
\begin{tabular}{c|cll|cll}
   \multicolumn{1}{c}{}    & \multicolumn{3}{c}{$n_c=8$} & \multicolumn{3}{c}{$n_c=5$} \\ \hline\hline
Iteration     & \# eigenvalue & residual & conv. rate &  \# eigenvalue & residual & conv. rate  \\
0   &    62   &   1.96$\times10^{-1}$ &                    & 65 & 2.06$\times10^{-1}$ &  \\
1   &    59   &   1.58$\times10^{-7}$ &                    & 59 &1.18$\times10^{-3}$ &   \\
2   &    59   &   8.78$\times10^{-11}$ & 1.8$\times10^{4}$  & 59 & 3.89$\times10^{-5}$ & 3.0$\times10^{1}$   \\
3   &    59   &   1.46$\times10^{-14}$ & 6.0$\times10^{4}$  & 59 & 6.48$\times10^{-8}$ & 6.0$\times10^{2}$  \\
4    &         &            &                             & 59 &1.30$\times10^{-11}$ & 5.0$\times10^{3}$   \\
5    &         &            &                             & 59 & 5.02$\times10^{-14}$ & 2.6$\times10^{2}$    \\ \hline \hline
\end{tabular}
\end{small}
\caption{\label{tab:feastruns} Convergence results obtained using
  FEAST v2.1 for a new interval $\ab$ ($a=1.38695..$ and
  $b=1.88290..$) containing 59 eigenvalues and where the subspace size
  has been estimated at $100$ by counting the eigenvalues between
  $[a-(b-a)/4, \ b+(b-a)/4]$.  Two runs using $n_c=8$ and $n_c=5$
  integration points are considered, and the convergence rates are
  provided from the moment the number of interior eigenvalues
  stabilizes.}
\end{center}
\end{table}
In particular, the results show a linear convergence rate for $n_c=8$
in agreement with the expected $4\times 10^{4}$ that can be directly
obtained from a reading of the data in Fig.~\ref{fig:gauss} (at
$\lambda=\pm1.5$).  Moreover, the value of the rational function using
$n_c=5$ indicates an expected convergence rate of $2\times 10^{2}$
while the reported rates in Table \ref{tab:feastruns} appears to be in
agreement or much better after few iterations.


\subsection{Rational approximation filtering for non-symmetric problems}
 
Since the eigenvalues of any projector $P$, whether orthogonal or not,
are equal to either zero or one, its trace is always equal to the
number of its nonzero eigenvalues, which in the case under
consideration is the number of eigenvalues located inside the contour
integral. The stochastic estimator also works for non-symmetric
matrices. Therefore, the whole technique extends to non-Hermitian case
without any difference. The contour integration approach for computing
the eigenpairs in a given region of the complex plane has also been
successfully utilized within the framework of the FEAST solver
\cite{laux}.

Here we present some preliminary results on the applicability of the
eigenvalue count estimate (\ref{eq:estR}) for non-symmetric problems.
In order to illustrate the accuracy of this 
approach, we select the complex symmetric matrix 'qc324' available
from the Matrix Market
\footnote{\url{http://math.nist.gov/MatrixMarket/}} which is of size
$n = 324$. Using a circle centered at (0,0) with radius 0.04, the
exact (complex) eigenvalue count is 37. In order to estimate the trace
${\trace(P)}$ one generates complex random vectors $v_k, k=1,..,n_v$
with entries $\pm 1$ for both the real and imaginary parts. The
imaginary part of the trace is expected to be very small, so the
estimates are obtained using the real part of the trace in our
simulation results.  Fig.~\ref{fig:qc324} presents the estimates
obtained using two different contour integrations.
\begin{figure}[hbt]
\centerline{ 
\includegraphics[width=0.5\textwidth,height=0.4\textwidth]{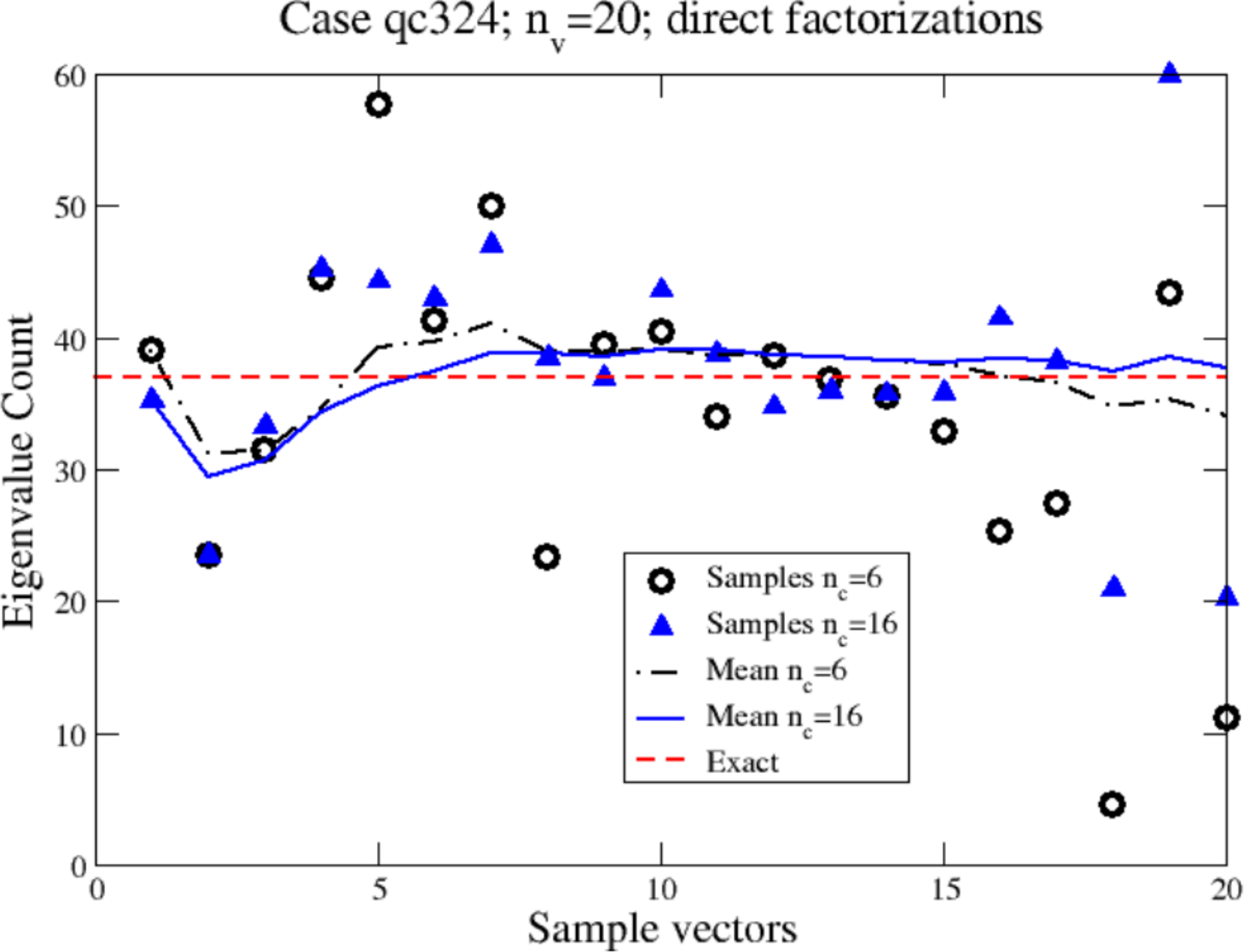}} 
\caption{The rational approximation method at work for the qc324 
matrix using a direct solver.
The plots compare a run using 6 and 16 integration points for the 
whole contour (obtained by placing 
respectively 3 and 8 Gauss integration points in the half-circle) and 20
sample vectors. 
\label{fig:qc324}}
\end{figure}
One notes that $n_c=6$ (i.e. Gauss-3 for each half-circle) already 
provides some reasonable estimates.
Further work would be needed to report a detailed study of the accuracy of the approach, but 
this preliminary result shows promise on the potential extension of
the estimated eigenvalue counts to the complex plane.

\section{Conclusion}\label{sec:concl} 
The methods  presented in this paper rely on a compromise
between accuracy and speed. If one is interested in an exact
count, then clearly combining  the Sylvester inertia theorem
with some direct solution method may be the best option, although
this may be too costly or even impractical in some situations.
Otherwise, a stochastic estimate based on approximating the trace
of the eigen-projector by exploiting a polynomial or rational function
expansion, will be sufficient.
The rational approximation viewpoint may be perfectly suitable 
within the context  of a package like FEAST since 
(approximate) factorizations will be needed at the outset anyway.
The initialization of the package will begin by estimating
the eigenvalue count in order to determine the proper subspace 
size $M_0$ to use. The additional cost of this step then remains
relatively small and its use may lead to great savings.
In other applications the polynomial approximation can give
a good estimate at a relatively low cost.

Since the  proposed methods are all  based on an  approximation of the
spectral  projector,  they  are  subject  to a  slight  bias  if  this
approximation is not accurate enough.  This bias may be exacerbated by
the presence of clusters near  the interval boundaries, since it is
generally near these locations that the inaccuracies of the 
approximate projector are large. 
Finding reliable methods to reduce this bias 
remains an open issue that is worth investigating.

\bibliographystyle{siam}

\bibliography{strings,saad,local}

\end{document}